\newif\ifpictures
\picturestrue

\documentclass[11pt]{amsart}
\usepackage{latex_base}
\usepackage{macros}
\usepackage{makecell} 

\author{Mareike Dressler}
\address{Mareike Dressler, School of Mathematics and Statistics, University of New South Wales, Sydney, NSW 2052, Australia.}
\email{m.dressler@unsw.edu.au}

\author{Qi Wang}
\address{Qi Wang, School of Mathematics and Statistics, University of New South Wales, Sydney, NSW 2052, Australia.}
\email{qi.wang21@student.unsw.edu.au}



\title[B-SOS+SONC Hierarchy]{A Bounded Degree SOS Plus SONC Hierarchy for Polynomial Optimization} 

\begin{document}

\begin{abstract}
We propose a bounded degree SOS+SONC hierarchy for constrained polynomial optimization, termed B-SOS+SONC. Starting from Lasserre's bounded-degree SOS framework, we enlarge the certificate cone from SOS to the recently introduced SOS+SONC cone, thereby combining the algebraic strength of semidefinite relaxations with the sparse structure captured by circuit polynomials. We show that, for each fixed certificate degree, the resulting hierarchy is complete, that is, its optimal values are monotone and converge to the global optimum. Moreover, we derive an explicit SDP-REP reformulation, so that each relaxation can be solved within a tractable convex optimization framework over semidefinite and relative entropy cones.

Beyond the optimization hierarchy itself, we investigate structural properties of the SONC cone and introduce the notions of first-order and second-order SONC-convexity. This leads to a new sufficient condition for first-level exactness of the B-SOS+SONC hierarchy. Numerical experiments illustrate that the proposed hierarchy often yields tighter lower bounds than the B-SOS relaxation while remaining tractable.
\end{abstract}

\maketitle

\section{Introduction}\label{sec: Introduction}
Polynomial optimization over basic closed semialgebraic sets is a fundamental problem at the intersection of real algebraic geometry, convex optimization, and applications. Given real, multivariate polynomials
\(f,g_1,\ldots,g_m \in \R[\xb]\), one aims to compute
\[
        f_K^{\star} = \inf\{ f(\xb) : \xb \in K\} , \text{ with }
        K = \{\xb \in \R^n : g_j(\xb) \geq 0,\ j=1,\ldots,m\}.
\]
It is well-known that solving such problems is in general NP-hard. One approach involves constructing a sequence of  tractable convex relaxations, a so-called hierarchy, that converges to the global optimum.  The quality and efficiency of such relaxations depend crucially on the underlying certificate of nonnegativity.

A classical nonnegativity certificate is given by \emph{sums of squares (SOS)}. Building on \emph{Putinar's Positivstellensatz}~\cite{putinar1993positive}, the SOS-hierarchy \cite{lasserre2001global,Parrilo:PhD}, often referred to as \emph{Lasserre's hierarchy},  provides a systematic sequence of \emph{semidefinite programming (SDP)} relaxations whose optimal values converge to the global optimum under standard compactness assumptions. In fact, the convergence is generically finite~\cite{Nie:FlatTruncation,Nie:FiniteConvergence}. Although this relaxation comes with strong convergence guarantees, its practical limitations are well known.  Each level in the relaxation requires solving an SDP involving positive semidefinite matrices of size $\binom{n+d}{d}$, thus quickly becoming too large as either the number of variables $n$ or the degree $d$ increases.

A different route is offered by \emph{linear programming (LP)} relaxations based on the \emph{Krivine-Stengle Positivstellensatz}~\cite{Krivine,Stengle}, see also Theorem~\ref{thm:Krivine-Stengle}. After a suitable normalization \(0 \leq g_j(\xb) \leq 1\) on \(K\), positivity on \(K\) can be certified by nonnegative combinations of products
\[
       g_1^{a_1}\cdots g_m^{a_m}
        (1-g_1)^{b_1}\cdots (1-g_m)^{b_m}.
\]
Truncating this certificate yields a hierarchy of linear programs. From a computational perspective, replacing semidefinite programs by linear programs is highly attractive. However, this improvement comes at a price: the resulting certificates are much more restrictive, finite convergence cannot be achieved for most convex problems~\cite{Lasserre:SDPvsLP,Lasserre:Book-Moments}, and the corresponding linear programs can be numerically delicate.

\smallskip
In~\cite{lasserre2013lagrangian}, Lasserre proposed the \emph{bounded degree SOS (B-SOS) hierarchy} to circumvent these issues. It relaxes Krivine-Stengle's certificate to an SOS condition of a fixed, prescribed degree. That is, it keeps the size of the associated SDP manageable, by a priori fixing their degree to a specific parameter $k$. In this way, the hierarchy still progresses through the Krivine-Stengle products, while the size of the semidefinite block is fixed in advance by the user. This combines two desirable features: the relaxation keeps a bounded SDP component, yet retains substantially more expressive power than the LP hierarchy. Moreover, it exhibits finite convergence at the first level for important classes of convex polynomial optimization problems, for instance under SOS-convexity assumptions. The effective implementation of this approach can be found in~\cite{lasserre2017bounded} and additional sparsity was considered in~\cite{weisser2018sparse}.  This bounded-degree viewpoint has since led to further variants, bounded-degree hierarchies based on \emph{scaled-diagonal dominant sums of squares (SDSOS)}  and \emph{second order cone programming (SOCP) } relaxations~\cite{chuong2019new}, and recently based on \emph{separable plus lower degree (SPLD) polynomials}~\cite{jiao2025spld}.

\medskip
In parallel, \emph{sums of nonnegative circuit polynomials (SONC)} have emerged as an alternative certificate of polynomial nonnegativity, see~Section~\ref{subsubsec:SONC}.
Formally introduced in~\cite{iliman2016amoebas}, SONC certificates build on Reznick's agiforms~\cite{Reznick:AGI} and are closely related to independent approaches in~\cite{Fidalgo:Kovacec,PanteaEtA:GlobalInjectivityAndMultipleEquilibria}, as well as to the \emph{SAGE cone} for signomials~\cite{ChandrasekaranShah:REPforSignomialOptimization}. In contrast to SOS, which is algebraic in nature and tied to Gram matrix representations, SONC certificates are governed by the geometry of the Newton polytope and by circuit number inequalities. This makes them particularly suitable for sparse polynomials. Membership in the SONC cone can be formulated via \emph{relative entropy programming (REP)}~\cite{dressler2017positivstellensatz}, and SONC-based methods have been successfully used in polynomial optimization. The SONC approach can also be formulated using  \emph{geometric programming}~\cite{IlimanDeWolff_LowerBounds}, \emph{second order cone programming}~\cite{MagronWang_SecondOrderConeProgramming}, or \emph{duality theory}~\cite{Papp_DualitySONC}. Positivstellensatz-type results for SONC and SAGE further show that these certificates give rise to complete hierarchies for constrained polynomial and signomial optimization; see, for example,~\cite{dressler2017positivstellensatz,WJYP2020,Dressler:Murray}.

Importantly, SOS and SONC are independent certificates in general: neither cone contains the other outside the classical Hilbert cases. Consequently, their Minkowski sum, the \emph{SOS+SONC cone}, provides a strictly richer nonnegativity certificate than either SOS or SONC alone. Recent work has initiated the systematic study of this combined cone~\cite{dressler2025study} and its use in unconstrained polynomial optimization~\cite{Schick2025-01-31squar-72682}. 

\smallskip
This suggests a natural question: can the bounded degree SOS philosophy be combined with the additional expressive power of SONC certificates? The purpose of this paper is to answer this question affirmatively. We introduce a \emph{bounded degree SOS+SONC (B-SOS+SONC) hierarchy} for constrained polynomial optimization, see Section~\ref{Sec: The bounded-SOS+SONC hierarchy and finite convergence}. Starting from the Krivine-Stengle Lagrangian used in the B-SOS hierarchy, we replace the fixed degree SOS certificate by a fixed degree SOS+SONC certificate. Thus, the resulting relaxation preserves the bounded degree structure of B-SOS, while enlarging the certificate cone by allowing a SONC component. Computationally, this leads to a convex program over the product of semidefinite and relative entropy cones. Conceptually, it gives a bounded degree hierarchy that can exploit both algebraic SOS structure and sparse circuit structure. 

\medskip
Our contributions are as follows. 
First, we define the bounded SOS+SONC relaxation for constrained polynomial optimization and show that it forms a complete hierarchy (Theorem~\ref{thm:hierarchyB-SOS+SONC}). For every fixed certificate degree, the resulting sequence of lower bounds is monotone and converges to the global optimum. Since the SOS cone is contained in the SOS+SONC cone, the new hierarchy yields bounds that are at least as strong as those of the corresponding B-SOS hierarchy.

Second, we derive an explicit SDP-REP formulation of the relaxation. This shows that the B-SOS+SONC relaxation remains tractable within conic convex optimization: the SOS part is represented by a semidefinite constraint, while the SONC part is represented by relative entropy constraints. We also discuss an implementation and compare the resulting bounds with those obtained from B-SOS relaxations.

Third, we investigate finite convergence at the first level. The new hierarchy inherits the known first-level exactness result under SOS-convexity. More importantly, the SONC component leads us to introduce and study SONC-based convexity notions, which is our main theoretical contribution. We distinguish first-order and second-order SONC-convexity, establish structural closure properties of the SONC cone needed for their analysis, and prove that first-order SONC-convexity yields a new sufficient condition for first-level exactness. This provides a SONC-driven mechanism for finite convergence which is genuinely different from the classical SOS-convexity argument.

\medskip
The paper is organized as follows. Section~\ref{sec:Preliminaries} recalls basic notation,  necessary background on nonnegativity certificates, and the bounded degree SOS relaxation. 
In Section~\ref{Sec: The bounded-SOS+SONC hierarchy and finite convergence}, we introduce the bounded degree SOS+SONC relaxation for a constrained polynomial optimization problem, prove its convergence, and derive its SDP-REP formulation.
Section~\ref{Sec: SONC-convexity} is devoted to SONC-convexity:  we study structural invariance properties of the SONC cone, introduce the novel concept of SONC-convexity, where we distinguish between first-order and second-order SONC-convexity, and analyze their relationship. Further, we establish the corresponding first-level exactness result, see Section~\ref{subsec: 1level_exactness_SONC-conv}. 
In Section~\ref{sec:numerics}, we discuss the implementation and present numerical experiments. 
We conclude with open questions and possible directions for further work in Section~\ref{sec:conclusion}.

\section{Preliminaries} \label{sec:Preliminaries}

For $m \in \N$, we write $\struc{[m]}\coloneqq\{1,\ldots, m\}$ for the discrete set containing all integers from $1$ to $m$ and we set $\struc{[0]}\coloneqq\emptyset$. Further, we use $\struc{|\alpb|} \coloneqq \sum_{i=1}^n \alpha_i$, for $\alpha \in \N^n$ and $\struc{\N_s^n} \coloneqq \{\alpb \in \N^n: |\alpb| \leq s\}$. 

The \emph{convex hull} of a finite set $A \subseteq \N^n$, is denoted by $\struc{\conv(A)}$, and we write $\struc{V(A)}$ for the set of \emph{vertices} of $\conv(A)$. 

Let $\struc{\R[\xb]}\coloneqq \R[x_1,x_2,\dots,x_n]$ denote the ring of real polynomials in n variables, and let $\struc{\R[\xb]_{n,d}}$ be the vector space of polynomials in $n$ variables of degree at most $d$. 

The \emph{support} of a polynomial $f(\xb) = \sum_{\alpb} f_{\alpb} \xb^{\alpb}$, where $\xb^{\alpb} \coloneqq x_1^{\alpha_1}\cdots x_n^{\alpha_n}$, is defined as $\struc{\supp(f)} := \{\alpb \in \N^n: f_{\alpb} \neq 0\}$,
and its \struc{\emph{Newton polytope}} is $\struc{\New(f)} \coloneqq \conv(\supp(f))$. 
A term $f_{\alpb}\xb^{\alpb}$ is called a \struc{\emph{monomial square}} if $f_{\alpb} \in \R_{\geq 0}$ and $\alpb \in 2\N^n$. We refer to a polynomial $f$ as \struc{\emph{separable}} if it admits a decomposition $f(\xb)=\sum_{i=1}^n f_i(x_i)$. Specifically in the context of sparsity, we oftentimes consider polynomials supported on a specific finite set $\cA\subseteq \N^n$, and denote by $\struc{\R[\cA]}$ the vector space of real polynomials supported on  $\cA$. For a polynomial $f\in \R[\cA]$, we set $\struc{\cA(f)} \coloneqq \New(f) \cap \N^n$ and $\struc{\mathcal{H}(f)} \coloneqq \frac{1}{2}\cA(f) \cap \N^n$.

We use the superscript $\struc{\star^h}$ to indicate that the object $\star$ (typically polynomial or cone) is regarded in the homogeneous setting. 

Let $\struc{\Sc^n}$ denote the space of  $n \times n$ real symmetric matrices and write $Q \succeq 0$ if $Q$ is  \struc{\emph{positive semidefinite (PSD)}}. 
The \struc{\emph{relative entropy function}} is defined as follows:
\begin{align*}
        D\colon \R_{\geq 0}^m \times \R_{\geq 0}^m \to \R\cup\{+\infty\}, \quad \struc{D(\nub,\mub)} = \sum_{i=1}^m \nu_i \log\left( \frac{\nu_i}{\mu_i}\right).
\end{align*}
By convention, we set
\begin{equation*}
    \nu_i \log\left( \frac{\nu_i}{\mu_i}\right) \coloneqq 
    \begin{cases}
        0,\; \text{ if } \nu_i = 0,\\
        \infty,\; \text{ if } \nu_i > 0 \text{ and }\mu_i = 0,\\
        \nu_i \log\left( \frac{\nu_i}{\mu_i}\right),\; \text{otherwise}.
    \end{cases}
\end{equation*}

\subsection{Polynomial optimization}\label{subsec:PolyOpt}

We consider the constrained polynomial optimization problem
\begin{equation}\label{eq:POP}
\struc{f_K^\star} \coloneqq \inf_{\xb \in K} f(\xb),
\end{equation}
where $f \in \R[\xb]$ and the feasible set $K \subseteq \R^n$ is a basic closed semialgebraic set of the form
\[
\struc{K} \coloneqq \{\xb \in \R^n : g_j(\xb) \ge 0,\; j=1,\dots,m\},
\]
with $g_j \in \R[\xb]$.
A standard reformulation expresses \eqref{eq:POP} as a nonnegativity problem:
\begin{align}\label{eq:POP-nn-problem}
f_K^\star = \sup \big\{ t \in \R : f(\xb) - t \ge 0 \ \text{for all } \xb \in K \big\}.
\end{align}
Thus, solving \eqref{eq:POP} amounts to certifying nonnegativity of polynomials over $K$.
A classical approach to such certificates is provided by various \emph{Positivstellensätze} often invoking obvious nonnegative polynomials (see Section~\ref{subsec:NNcertificates}). 
These representations form the theoretical basis of several hierarchies of convex relaxations for polynomial optimization.

\smallskip
In particular, we recall the Positivstellensatz by  Krivine-Stengle~\cite{Krivine,Stengle}.
\begin{theorem}\label{thm:Krivine-Stengle}
Assume that $0\leq g_j(\xb)\leq 1$ on $K$ for every $j$ and the family $\{g_j(\xb),1-g_j(\xb)\}$ generates $\R[\xb]$. If $f(\xb)$ is strictly positive on $K$, then 
\[
f = \sum_{(\ab,\bb) \in \N^{2m}} \lam_{\ab\bb} \prod_{j=1}^m \left(g_j^{a_j} (1-g_j)^{b_j}\right),
\]
for nonnegative scalars $(\lam_{\ab\bb})$. 
\end{theorem}
Relaxing the nonnegativity condition in~\eqref{eq:POP-nn-problem} to a Krivine-Stengle certificate of positivity yields a linear program.

\subsection{Tractable nonnegativity certificates}\label{subsec:NNcertificates}
Since deciding nonnegativity of polynomials is NP-hard, one replaces it by membership in tractable cones that provide sufficient certificates.

\subsubsection{Sums of squares (SOS)}
A polynomial $f \in \R[\xb]_{n,2d}$ is a \struc{\emph{sum of squares (SOS)}} if there exist $g_1,\dots,g_k \in \mathbb{R}[\xb]_{n,d}$ such that $f = \sum_{i=1}^k g_i^2$.
The cone of SOS polynomials in $\R[\xb]_{n,2d}$ is denoted by $\struc{\Sigma_{n,2d}}$, and, its restriction to a support set $\cA$ by
\[
\struc{\Sigma_{n,2d}(\cA)} \coloneqq \{f \in \Sigma_{n,2d} : \supp(f) \subseteq \cA\}.
\]
Membership in $\Sigma_{n,2d}$ can be decided via \struc{\emph{semidefinite programming (SDP)}} through a Gram matrix representation, see e.g.~\cite{Laurent:survey}.

\smallskip

If the Gram matrix in this representation is \emph{diagonally dominant} or \emph{scaled diagonally dominant}, this leads to so-called \struc{\emph{DSOS}} and  \struc{\emph{SDSOS polynomials}}, respectively. In this case, searching for such a decomposition reduces to linear programming and second-order cone programming. For further details, we refer to~\cite{Ahmadi:Majumdar:SOS-SDSOS}.  

\medskip
In the SOS framework, we have for instance \struc{\emph{Lasserre's hierarchy}}~\cite{lasserre2001global}, which provides SOS/SDP relaxations for~\eqref{eq:POP} based on \emph{Putinar's Positivstellensatz}~\cite{putinar1993positive}. 

\subsubsection{Sums of nonnegative circuit polynomials (SONC)}\label{subsubsec:SONC}
An alternative certificate of nonnegativity especially suitable for sparsity are circuit polynomials. 

Recall, a finite set $A \subseteq \N^n$ is called a \struc{\emph{circuit}} if it is minimally affinely dependent, i.e., $A$ is affinely dependent but every proper subset is affinely independent, see e.g. \cite{GKZ:Discriminants}. We focus on the case where $A = S \cup \{\betab\}$, with $S$ affinely independent and $\betab$ contained in the relative interior of $\conv(S)$. 
Such a configuration is often referred to as a \emph{simplicial circuit}. 
For simplicity, we will refer to these sets simply as circuits and denote them by $(S,\beta)$.

Since $\conv(S)$ is a simplex, the point $\betab$ admits a unique representation as a convex combination of the elements of $S$, i.e., $\betab = \sum_{\alpb \in S} \gamma_{\alpb} \alpb$, with $\gamma_{\alpb} >0$ for all $\alpb \in S$ and $\sum_{\alpb \in S} \gamma_{\alpb} = 1$. 
The coefficients $\struc{(\gamma_{\alpha})_{\alpha \in S}}$ are called the \struc{\emph{barycentric coordinates}} of $\betab$ (with respect to $S$).

\medskip
A polynomial $f \in \R[\cA]$ is called a \struc{\emph{circuit polynomial}} if it is either a (sum of) monomial square(s) or it is supported on a circuit $(S,\betab)$ with $S \subseteq 2\N^n$. Therefore, it can be written as
\begin{equation}\label{eq:defcircuit}
f(\xb)=\sum_{\alpb \in S} c_{\alpb} \xb^{\alpb} + d \xb^{\betab},
\end{equation}
where $c_{\alpb} > 0$ for all $\alpb \in S$ and $d \in \R$.
For every circuit polynomial $f$, its associated \struc{\emph{circuit number}} is 
\begin{align*}
 \struc{\Theta_f}  \coloneqq \prod_{\alpb \in S} \left( \frac{c_{\alpb}}{\gamma_{\alpb}}\right)^{\gamma_{\alpb}}.
\end{align*}
A key property is that nonnegativity of $f$ can be characterized explicitly in terms of $\Theta_f$: 
the polynomial $f$ is nonnegative on $\R^n$ if and only if either it is a sum of monomial squares or $|d| \leq \Theta_f$, see~\cite[Theorem 1.1]{iliman2016amoebas}.

A polynomial is called a \struc{\emph{sum of nonnegative circuit polynomials (SONC)}} if it can be expressed as a conic combination of nonnegative circuit polynomials. We denote the corresponding cone of all SONC polynomials in $\R[\xb]_{n,2d}$ by $\struc{C_{n,2d}}$, and by $\struc{C_{n,2d}(\cA)}$ its restriction to polynomials supported on $\cA$. In contrast to SOS, SONC polynomials always admit a \emph{cancellation-free decomposition} into nonnegative circuits whose supports are contained in the Newton polytope (see, e.g., \cite[Section 5]{wang2022nonnegative} and \cite[Section 5]{murray2021newton}); we call this feature sparsity-preserving property.

\smallskip

Membership in $C_{n,2d}$ can be checked via \struc{\emph{relative entropy programming (REP)}}, see e.g.,~\cite{dressler2017positivstellensatz}. The SONC cone also allows for optimization approaches via  \emph{geometric programming} \cite{IlimanDeWolff_LowerBounds}, \emph{second order cone programming}  \cite{MagronWang_SecondOrderConeProgramming}, or \emph{duality theory} \cite{Papp_DualitySONC}.

\smallskip
Lastly, we point out that sums of binomial squares are SONC polynomials, and since SDSOS polynomials can be interpreted as sums of binomial squares (with respect to the monomial basis), they are also SONC.

\subsubsection{SOS+SONC polynomials}\label{subsubsec:Prelim_SOS+SONC}
As the SOS and SONC cones are in general independent of each other, a natural, recent approach is to consider their combination. 

Following~\cite{dressler2025study}, we say a polynomial $f\in \R[\xb]_{n,2d}$ is a \struc{\emph{sum of squares plus a sum of nonnegative circuit polynomials (SOS+SONC)}}, if it decomposes as $f=f_{\sos}+f_{\sonc}$ with $f_{\sos}\in \Sigma_{n,2d}$ and $f_{\sonc} \in C_{n,2d}$. The set of all SOS+SONC polynomials is denoted by $\struc{(\Sigma+C)_{n,2d}}$. 

\smallskip
Note that, $(\Sigma+C)_{n,2d}=\Sigma_{n,2d} +C_{n,2d}=\conv\left(\Sigma_{n,2d} \cup C_{n,2d}\right)$. 
With $\struc{P_{n,2d}}$ being the cone of nonnegative $n$-variate polynomials of degree at most $2d$, we further have 
$(\Sigma_{n,2d}  \cup C_{n,2d} ) \subseteq (\Sigma+C)_{n,2d}  \subseteq P_{n,2d}$, 
where the inclusions are strict in all non-Hilbert cases (\cite{dressler2025study}).

\smallskip
Deciding membership in $(\Sigma+C)_{n,2d}$ amounts to an SDP-REP feasibility problem. 
An \struc{\emph{SDP-REP (problem)}} is a conic problem over the Cartesian product of the semidefinite and the relative entropy cones~\cite{Schick2025-01-31squar-72682}. We can write such problems as
\begin{align}\label{eq:SDP-REP-def}
\min_{Q\in \Sc^{n};\, \nub,\mub\in (\R^m)^k;\,\cb\in \R^k} \quad & f(Q,\nub,\mub,\cb) \notag \\
\text{s.t.} \quad 
&(1)\;\; Q \succeq 0, \notag \\
&(2)\;\; D(\nub^{(i)},\mub^{(i)})\leq \cb^{(i)}, \quad (i \in [m]) \notag \\
&(3)\;\; h_j(Q,\nub,\mub,\cb) =  0, \quad (j \in [\ell]) \notag 
\end{align}
with $n,m,k,l\in \N$, and $f$ and $h_j$ being linear and affine linear functions, respectively. 
Constraint~(1) encodes the SDP constraint and (2) the REP one. 

\smallskip
The situation for the support restricted SOS+SONC cone is slightly more elaborate. Let $\struc{(\Sigma+C)_{n,2d} (\cA)} \coloneqq \{f\in (\Sigma+C)_{n,2d}: \supp(f) \subseteq \cA\}$. Then, 
\begin{equation}\label{eq:SOS+SONC_A-restricted_decomp}
f\in (\Sigma+C)_{n,2d} (\cA)\; \Leftrightarrow\; f_{\sos}\in \Sigma_{n,2d}(\cA(f)), \,f_{\sonc}\in C_{n,2d}(\cA(f)) \text{ s.t. } f=f_{\sos}+f_{\sonc}. 
\end{equation}
Since $\cA(f)\subseteq \conv(\cA)\cap \N^n$, we have $|\cA(f)|\geq |\supp(f)|$ and, in general, $|\cA(f)|\geq |\cA|$. Thus, we usually need to enlarge the supports of $f_{\sos}$ and $f_{\sonc}$ in comparison to $\supp(f)$, when searching for a decomposition~\eqref{eq:SOS+SONC_A-restricted_decomp}. A detailed discussion on this can be found in Sections 5.4.1 and 5.4.2 of~\cite{Schick2025-01-31squar-72682}. For an explicit SDP-REP formulation for membership in $(\Sigma+C)_{n,2d} (\cA)$, we refer to~\cite[Corollary 5.4.9]{Schick2025-01-31squar-72682}.

\subsection{The bounded-SOS relaxation}\label{subsec: b-sos}

Now we recall the bounded degree SOS (B-SOS) relaxation, originally proposed in \cite{lasserre2013lagrangian} and further developed in \cite{lasserre2017bounded}. 
Throughout the paper, we impose the following assumption, which is standard in the analysis of bounded degree relaxations.

\textbf{Assumption A: } 
The set $K$ is nonempty and compact, and the family of polynomials $\{g_1,\ldots,g_m,1-g_1,\ldots,1-g_m\}$ generates the algebra $\R[\xb]$.

Since $K$ is compact, each $g_j$ is bounded on $K$, and by rescaling if necessary, we may assume without loss of generality that
$0 \leq g_j(\xb) \leq 1$ on $K$ for all $j = 1,\ldots, m$.

Moreover, Assumption A is not restrictive: if $K$ is compact, it can always be enforced by a suitable reformulation.
Indeed, after rescaling the variables, one may assume $K\subseteq [0,1]^n$, and by adding redundant constraints $0\leq x_i\leq 1, i=1,\ldots,n$, the above family includes $x_i$ and $1-x_i$, and therefore generates $\R[\xb]$.

\smallskip
For fixed $d \in \N$ and nonnegative \struc{\emph{Lagrange multipliers}} $\struc{\lamb = (\lam_{\ab \bb})}$, we introduce the  \struc{\emph{Lagrange polynomial}}:
\begin{align}\label{eq:Lagrangian-function}
    \xb \mapsto \struc{L_d(\xb, \lamb)} \coloneqq f(\xb) - \sum_{(\ab,\bb) \in \N_d^{2m}} \lam_{\ab\bb} \prod_{j=1}^m g_j^{a_j}(\xb) (1-g_j(\xb))^{b_j}.
\end{align}
Recall, $(\ab,\bb) \in \N_d^{2m}=\{\ab \in \N^m, \bb \in \N^m : |\ab| + |\bb| \leq d\}$.

We now consider a family of optimization problems indexed by $d\in \N$. For this, we fix $k\in \N$ and define the \struc{bounded degree SOS (B-SOS) relaxation}:
\begin{align}
    \label{eq:B-SOS-hierarchy}
    \struc{q_d^k} \coloneqq \sup_{t} \{t\in \R\quantify L_d(\xb,\lamb) - t  \in \Sigma_{n,2k},\, \lamb \geq 0\}.
\end{align}
Clearly, computing each $q_d^k$ reduces to solving an SDP, and in fact~\eqref{eq:B-SOS-hierarchy} leads to a complete hierarchy of SDP relaxations for the constrained optimization problem~\eqref{eq:POP}. The computational feature of this approach comes from the fact, that the associated matrices to check the semidefinite constraint in~\eqref{eq:B-SOS-hierarchy} are of fixed size $\binom{n+k}{k}$, so independent of the relaxation order $d\in \N$.

\smallskip
In the case, that Slater's condition holds, and the polynomials $f$ and $-g_j$, $j \in [m]$, are \emph{SOS-convex} of degree at most $2k$, then we have $q_1^k = f_K^{\star}$, i.e., finite convergence takes place at the first relaxation in the hierarchy (first-level exactness). 

Remember, that a polynomial is said to be \struc{\emph{SOS-convex}} if its Hessian $H(\xb)$ is an \emph{SOS-matrix}, i.e., it factors as $M(\xb)M(\xb)^\top$ for some rectangular matrix polynomial $M$. This is equivalent to $\yb^\top H(\xb)\yb$ being an SOS in $\R[\xb;\yb]$.


\section{The bounded SOS+SONC hierarchy and finite convergence}\label{Sec: The bounded-SOS+SONC hierarchy and finite convergence}

In this section, we introduce a new hierarchy for constrained polynomial optimization by combining the bounded degree SOS framework with the SOS+SONC cone. The resulting hierarchy preserves the tractability of bounded SOS relaxations while strictly enlarging the certificate space through the additional expressive power of SONC polynomials. 
This allows us to capture nonnegativity structures that are inaccessible to SOS certificates alone, without sacrificing convergence guarantees.

\subsection{The bounded SOS+SONC hierarchy}
Let Assumption A hold.

In \cite{dressler2022hypercube} it is shown that, on the Boolean hypercube $\{0,1\}^n$, nonnegativity certificates based on the SONC cone can be used in place of SOS certificates. In fact, in this setting the SONC cone is contained in the SOS cone; see \cite{Kurpisz:deWolff}.
In the present framework, however, the constraints $0 \leq g_j \leq 1$ for $j=1,\ldots,m$ do not imply that the feasible region $K$ has a hypercube structure. Rather, $K$ is a general compact semialgebraic set, and the cones of SOS and SONC polynomials are in general incomparable.

Consequently, neither SOS nor SONC certificates alone fully capture the strengths of the other. This motivates the use of the combined SOS+SONC cone, which allows one to exploit their complementary features in constructing bounded degree relaxations.
Our construction can also be viewed as a natural extension of the bounded-degree hierarchy based on SDSOS and SOCP relaxations~\cite{chuong2019new}, in the sense that it keeps the bounded-degree philosophy while enlarging the certificate cone from SDSOS to SOS+SONC.
\medskip

Let $\cA \subseteq \N^n$ be a finite support set. For fixed $k \in \N$ and parameter $d \in \N$, we define the \struc{bounded degree SOS+SONC  (B-SOS+SONC) relaxation } by
\begin{equation}\label{eq: B-SOS+SONC-hierarchy}
\struc{p_{d,\cA}^k} \coloneqq \sup \left\{ t \in \R \quantify L_d(\xb,\lamb) - t \in (\Sigma + C)_{n,2k}(\cA),\; \lamb \geq 0 \right\},
\end{equation}
where $L_d(\xb,\lamb)$ is as in~\eqref{eq:Lagrangian-function}. 
We emphasize that $k$ is a fixed constraint size, while $d$ indexes the hierarchy level, so, the relaxation order. 
Observe that if the ambient support is $\cA=\N^{n}_{2k}$, so the support restriction is vacuous, the relaxation can be regarded in the support unrestricted SOS+SONC cone $(\Sigma + C)_{n,2k}$. In this case we simply write $\struc{p_d^k}\coloneqq p_{d,\N^{n}_{2k}}^k$ for the full-support hierarchy.

Every feasible value $t$ in~\eqref{eq: B-SOS+SONC-hierarchy} is a valid lower bound for
\eqref{eq:POP}. Indeed, if $(t,\lamb)$ is feasible, then for every
$\xb\in K$,
\begin{align*}
f(\xb)-t
&=
L_d(\xb,\lamb)-t
+
\sum_{(\ab,\bb)\in\N_d^{2m}}
\lambda_{\ab\bb}
\prod_{j=1}^m
g_j(\xb)^{a_j}(1-g_j(\xb))^{b_j}\geq0,
\end{align*}
because $L_d(\xb,\lamb)-t$ is globally nonnegative and
$0\leq g_j(\xb)\leq1$ on $K$. Consequently,
\[
p_{d,\cA}^k\leq f_K^\star.
\]

The constraint $L_d(\xb,\lamb)-t \in (\Sigma+C)_{n,2k}(\mathcal{A})$ requires that the polynomial $\struc{\cL}\coloneqq L_d(\xb,\lamb)-t$ admits a decomposition into a sum of a truncated SOS polynomial and a truncated SONC polynomial supported on $\mathcal{A}(\cL)$, see Section~\ref{subsubsec:Prelim_SOS+SONC}, i.e.,
\[
\cL=L_d(\xb,\lamb)-t = f_{\sos} + f_{\sonc},
\quad f_{\sos} \in \Sigma_{n,2k}(\mathcal{A(\cL)}), \quad f_{\sonc} \in C_{n,2k}(\mathcal{A(\cL)}).
\]

\smallskip

For completeness, we note that in order for the above representation to be meaningful, $k$ is typically chosen sufficiently large so that the support of $L_d(\xb,\lamb)$ is contained in $\mathcal{A}$. In particular, this is ensured whenever $k \ge \frac{1}{2}\max\{\deg(f),\, d \max_j \deg(g_j)\}$. 
We will return to this choice in the computational formulation in Section~\ref{sec:numerics}.

\begin{remark}[The case $k=0$]
For $k=0$, the B-SOS+SONC relaxation reduces to a linear program. In this case, the $d$-th relaxation of the hierarchy coincides with the classical LP relaxation based on the Krivine--Stengle Positivstellensatz, see Section~\ref{subsec:PolyOpt}.
\end{remark}

Next, we show that the B-SOS+SONC relaxation yields a complete hierarchy indexed by $(d,k)$, where $d$ controls the Lagrangian lifting and $k$ the a priori fixed polynomial truncation degree.

\begin{theorem}[Complete Hierarchy]
\label{thm:hierarchyB-SOS+SONC}
For every fixed $k \in \mathbb{N}$, the sequence $\{p_{d,\cA}^k\}_{d \in \mathbb{N}}$ is monotonically increasing and satisfies
\[
p_{d,\cA}^k \uparrow f_K^\star \quad \text{as } d \to \infty,
\]
where $f_K^\star$ denotes the optimal value of the underlying polynomial optimization problem over $K$.

In particular, the bounded SOS+SONC hierarchy is complete.
\end{theorem}

\begin{proof}
Monotonicity follows by extending every multiplier vector at level $d$ by zeros at level $d+1$.

\smallskip

It remains to show that the hierarchy is complete. 
Let $\eps>0$. Since
\[
f(\xb)-\bigl(f_K^\star-\eps\bigr)>0
\qquad\text{on }K,
\]
Theorem~\ref{thm:Krivine-Stengle} yields an integer
$d_\eps$ and nonnegative coefficients $\lambda_{\ab\bb}$ such that
\[
f(\xb)-\bigl(f_K^\star-\eps\bigr)
=
\sum_{(\ab,\bb)\in\N_{d_\eps}^{2m}}
\lambda_{\ab\bb}
\prod_{j=1}^m
g_j(\xb)^{a_j}(1-g_j(\xb))^{b_j}.
\]
Equivalently,
\[
L_{d_\eps}(\xb,\lamb)
-\bigl(f_K^\star-\eps\bigr)=0.
\]
Since the zero polynomial belongs to $(\Sigma+C)_{n,2k}(\cA)$, the value $f_K^\star-\eps$ is feasible for
$p_{d_\eps,\cA}^k$. By monotonicity,
\[
f_K^\star-\eps
\leq p_{d,\cA}^k
\leq f_K^\star
\qquad
\text{for all }d\geq d_\eps.
\]
Since $\eps>0$ was arbitrary,
$p_{d,\cA}^k\uparrow f_K^\star$.
\end{proof}

Theorem~\ref{thm:hierarchyB-SOS+SONC} shows that replacing the SOS cone by the larger SOS+SONC cone preserves the convergence guarantees of the B-SOS hierarchy.

Moreover, for the full-support choice $\cA=\N_{2k}^n$, the inclusion
$\Sigma_{n,2k}\subseteq(\Sigma+C)_{n,2k}$ gives
\begin{align}\label{ineq:comparisonof f^*, p_d^k and q_d^k}
q_d^k\leq p_d^k\leq f_K^\star.
\end{align}
Hence, the resulting bounds are at least as strong as those obtained from B-SOS relaxations.

\subsection{SDP-REP formulation} \label{subsec:SDP-REP formulation}

As before, let $\cA \subseteq \N^n$ be finite.
Recalling~\eqref{eq: B-SOS+SONC-hierarchy}, the bounded SOS+SONC relaxation can be written as
\begin{equation}
\label{eq:SDP-REP-start}
\begin{aligned}
\sup \quad & t \\
\text{s.t.} \quad 
& L_d(\xb,\lamb) - t = f_{\sos} + f_{\sonc}, \\
& f_{\sos} \in \Sigma_{n,2k}(\cA(\cL)), \quad f_{\sonc} \in C_{n,2k}(\cA(\cL)), \\
& \lamb \geq 0, \; t \in \R,
\end{aligned}
\end{equation}
where
\[
L_d(\xb,\lamb) = f(\xb) - \sum_{(\ab,\bb) \in \N_d^{2m}} \lambda_{\ab\bb} \prod_{j=1}^m g_j^{a_j}(1-g_j)^{b_j}.
\]

To derive a tractable formulation, we expand the polynomial $\cL=L_d(\xb,\lamb) - t$ as
\begin{equation}
\label{eq:lagrange-expansion}
L_d(\xb,\lamb) - t \eqqcolon \sum_{\etab \in \N_s^n} \struc{L_{\etab}(\lamb,t)}\,\xb^{\etab},
\end{equation}
where each coefficient $L_{\etab}(\lamb,t)$ depends affinely on $(\lamb,t)$, and the deterministic degree bound for each fixed $d$, $\struc{s} = \max\left\{\deg(f), \max_{(\ab,\bb) \in \N_d^{2m}} \deg\left(\prod_{j} g_j^{a_j}(1-g_j)^{b_j}\right)\right\}$.

\medskip
For $f\in \R[\cA]$, recall $\cA(f) = \New(f) \cap \N^n=\{\alpb(1),\dots,\alpb(|\cA(f)|)\}$ and $\mathcal{H}(f) = \frac{1}{2}\cA(f) \cap \N^n$. 
For a fixed index $j \in [m]$, we denote by $\struc{\fb_{\setminus j}} \coloneqq (f_{\alpb(i)})_{i \in [m]\setminus j}$ the vector of all coefficients except for the $j$-th.
Let $\struc{I} \coloneqq \{i \in [|\cA(f)|] \quantify \alpb(i) \notin 2\N^n\}$ be the index set of odd exponents in the support. 

Using the coefficient matching conditions together with the SOS and SONC representations, the B-SOS+SONC relaxation can be equivalently written as the following SDP-REP, for $2k \leq s$.

\begin{align}\label{eq:SDP-REP-final}
p_d^k = \sup_{\substack{t; \lamb;Q\in \Sc^{\abs{\cH(\cL)}}, \\ \big(\cb^{(l)}\big)_{l \in [|\cA(\cL)|]},\big(\nub^{(l)}\big)_{l \in [|\cA(\cL)|]} \R^{2\abs{\cA(\cL)}}}} \quad & t \notag \\
\text{s.t.} \quad 
&(1)\;\; L_{\etab}(\lamb,t) =
\sum_{\betab+\gamb=\etab} Q_{\betab,\gamb} + \sum_{l=1}^{|\cA(\cL)|} c^{(l)}_{\etab}
&& \forall \etab \in \N_{2k}^n, \notag \\
&(2)\;\; L_{\etab}(\lamb,t) = 0
&& \forall \etab \in \N_s^n,\; |\etab| > 2k, \notag \\
&(3)\;\; Q \succeq 0, \notag \\
&(4)\;\; \lamb \geq 0, \notag \\
&(5)\;\; \text{for each } l=1,\dots,|\cA(\cL)|,\; (\cb^{(l)},\nub^{(l)}) \text{ satisfy} \notag \\
&\quad (a)\;\; \cb_{\setminus l}^{(l)},\, \nub_{\setminus l}^{(l)} \geq 0, \notag \\
&\quad (b)\;\; \cb_{j}^{(l)} = \nub_{j}^{(l)} = 0 \quad (j \in I \setminus \{l\}), \notag \\
&\quad (c)\;\; D(\nub_{\setminus l}^{(l)}, e\,\cb_{\setminus l}^{(l)}) \leq \cb_l^{(l)}, \notag \\
&\quad (d)\;\; D(\nub_{\setminus l}^{(l)}, e\,\cb_{\setminus l}^{(l)}) \leq -\cb_l^{(l)} \quad (l \in I), \notag \\
&\quad (e)\;\; \sum_{i=1}^{|\cA(\cL)|} \alpb(i)\,\nub_i^{(l)} = 0, \notag \\
&\quad (f)\;\; \sum_{i=1}^{|\cA(\cL)|} \nub_i^{(l)} = 0. \notag
\end{align}

\medskip

The constraints can be interpreted as follows:
\begin{itemize}
    \item \textbf{(1)-(2)} enforce coefficient matching between the Lagrangian and the SOS+SONC decomposition;
    \item \textbf{(3)} encodes the SOS structure via a positive semidefinite Gram matrix;
    \item \textbf{(5)} imposes the SONC conditions through relative entropy constraints;
    \item \textbf{(4)} ensures nonnegativity of the Lagrange multipliers.
\end{itemize}

\smallskip

Hence, the B-SOS+SONC relaxation can be solved as a tractable convex optimization problem combining semidefinite and relative entropy programming; implementation details are deferred to Section~\ref{sec:numerics}.

\subsection{First-level exactness}
In this section, we identify a sufficient condition under which the B-SOS+SONC hierarchy achieves first-level exactness.

\begin{corollary}\label{cor:SOS-convex-first-level}
Assume that Slater's condition holds. If $f$ and $-g_j$, $j = 1,\ldots,m$, are SOS-convex polynomials of degree at most $2k$, then the B-SOS+SONC hierarchy is exact at the first level, that is,
\[
p_1^k = f_K^{\star}.
\]
\end{corollary}

\begin{proof}
Fixing $d = 1$ in~\eqref{ineq:comparisonof f^*, p_d^k and q_d^k} leads to $f_K^{\star} \geq p_1^k \geq q_1^k$. Under SOS-convexity of $f$ and $-g_j$ together with Slater's condition, the B-SOS relaxation is exact at level one, i.e., $q_1^k = f_K^{\star}$, see Section~\ref{subsec: b-sos}.
Consequently,
\[
f_K^{\star} \geq p_1^k \geq q_1^k = f_K^{\star},
\]
which implies $p_1^k = f_K^{\star}$.
\end{proof}

Corollary~\ref{cor:SOS-convex-first-level} shows that first-level exactness in the B-SOS+SONC hierarchy is inherited from the corresponding SOS-convexity mechanism of the B-SOS relaxation. In this regime, the SONC component is redundant for exactness, and the result is entirely driven by the SOS-convexity structure.

While this result follows from existing SOS-based arguments, it does not exploit the additional structure of the SONC cone. This motivates the introduction of new convexity notions tailored to the SOS+SONC framework.

In particular, this leads to the study of \emph{SONC-convexity conditions}, developed in Section~\ref{Sec: SONC-convexity}, as an alternative sufficient condition for first-level exactness.

A support-restricted extension of Corollary~\ref{cor:SOS-convex-first-level}, as well as of the SONC-based exactness results developed below, is presented in Remark~\ref{rem:support-restricted-exactness}.


\section{SONC-convexity} \label{Sec: SONC-convexity}

In this section, we introduce two SONC-based notions of convexity and show how they yield a new first-level exactness result for the B-SOS+SONC hierarchy. The key idea is to encode convexity directly within the SONC cone, thereby obtaining certificates that go beyond the SOS framework.

To start with, we derive structural properties of the SONC cone that will be used in the analysis of SONC-convexity; although the SONC cone is not invariant under general affine maps, we show it is indeed invariant under very structured monomial maps that preserve the combinatorial structure of supports, namely scaling, variable identification, and constant restriction. These algebraic, combinatorial stability properties are complemented by a limit, analytic stability property.

\subsection{Structural properties of the SONC cone} \label{subsec:further study of the SONC cone}

A key difference compared to the SOS cone is that the SONC cone is not invariant under general affine transformations of variables, see \cite{dressler2019approach}.
Nevertheless, it admits several favorable structured closure properties that are sufficient for our purposes.

\smallskip
We first collect several basic invariance properties of the SONC cone under structured variable transformations.

\begin{lemma}[Algebraic closure properties]
\label{lm:SONC_algebraic_closure}
The SONC cone $C_{n,2d}$ is preserved under the following operations:

\begin{enumerate}
\item[(i)] \textbf{Diagonal scaling:}  
Let $\ab \in \R^n$ with $a_i \neq 0$ for all $i$, and define the linear isomorphism
\[
T_{\ab} f(\xb) \coloneqq f(a_1 x_1,\ldots,a_n x_n).
\]
Then $T_{\ab}$ leaves $C_{n,2d}$ invariant.
\item[(ii)] \textbf{Nonnegative linear variable identification/projection:}  
The substitution $x_{p_i} = k_i x_{q_i}$ for $k_i \ge 0$ preserves SONC representations.

\item[(iii)] \textbf{Constant restriction:}  
Restricting $x_i = r_i$ for $r_i \in \mathbb{R}$ yields a SONC polynomial.
\end{enumerate}
\end{lemma}

\begin{proof}
(i) 
The map $\xb\mapsto (a_1 x_1,\ldots,a_n x_n)$ is a bijective monomial change of variables for any $a_i \neq 0$, hence nonnegativity is preserved. It suffices to consider circuit polynomials. Under $T_{\ab}$, supports are unchanged and coefficients rescale as $c_{\alpb} \mapsto c_{\alpb} \ab^{\alpb}$.
Using the barycentric relation $\betab = \sum_{\alpb} \gamma_{\alpb} \alpb$, the circuit number satisfies $\Theta_{T_{\ab}f} = \Theta_f |\ab|^{\betab}$, while the inner coefficient becomes $c_{\betab} \ab^{\betab}$.
Thus the circuit inequality is invariant under $T_{\ab}$, since  $|d| \leq \Theta_f$ if and only if  $|d \ab^{\betab}| \leq \Theta_f |\ab|^{\betab}$,
 and SONC membership follows. See also~\cite{averkov2021remark}.

(ii) The substitution $x_{p_i} = k_i x_{q_i}$ with $k_i \ge 0$ induces a monomial map on exponents. In particular, outer exponents remain even and affine dependencies defining the circuit are preserved, so circuit polynomials are mapped to SONCs (up to coefficient rescaling by nonnegative factors).
Nonnegativity is preserved under the substitution, hence each circuit polynomial is still a SONC. The claim follows by conic closure.

(iii) Setting $x_i = r_i$ maps each monomial to a scaled monomial in the remaining variables, preserving circuit structure and nonnegativity. Hence circuit polynomials remain SONC (of lower degree and in less variables), and by conic closure the claim follows.
\end{proof}

\medskip

In addition to these algebraic invariances, the SONC cone is also stable under limits of convex combinations, which will be useful in the analysis of variational representations.

\begin{lemma}[Closure under integration]
\label{lm:SONC_integration}
Let $\{f_t(\xb)\}_{t \in [0,1]} \subseteq \mathbb{R}[\xb]_{n,2d}$ satisfy $f_t \in C_{n,2d}$ for all $t \in [0,1]$. Assume moreover that
$t \mapsto f_t$ is Riemann integrable as a map into the finite-dimensional
coefficient space $\mathbb{R}[\xb]_{n,2d}$; equivalently, all coefficient
functions of $f_t$ are Riemann integrable. Then
\[
p(\xb) \coloneqq \int_0^1 f_t(\xb)\,dt \in C_{n,2d}.
\]
That is, the SONC cone is closed under integration of SONC-valued polynomial families over compact intervals.
\end{lemma}

\begin{proof}
For $N\in\mathbb{N}$, consider the uniform partition
\[
0=t_0<t_1<\cdots<t_N=1, \qquad t_i=\frac{i}{N}, \qquad \Delta t=\frac{1}{N}.
\]
Choose tags $t_i^\star\in[t_{i-1},t_i]$. The corresponding Riemann sums
\[
p_N(\xb) \coloneqq \sum_{i=1}^N f_{t_i^\star}(\xb)\Delta t
\]
are finite nonnegative combinations of SONC polynomials and hence belong to
$C_{n,2d}$, since $C_{n,2d}$ is a convex cone.
By the assumed Riemann integrability of $t\mapsto f_t$, the Riemann sums $p_N$ converge coefficient-wise to
\[
p(\xb)=\int_0^1 f_t(\xb)\,dt
\]
as $N\to\infty$. Since $C_{n,2d}$ is closed, see e.g.~\cite[Proposition  3.1.1.]{dressler2018sums}, the limit also belongs to~$C_{n,2d}$. Thus $p\in C_{n,2d}$.
 \end{proof}

\medskip

These structural properties will be used repeatedly in the development of SONC-based convexity notions, which we introduce next.

\subsection{SONC-convexity}\label{subsec:SONC-convexity}

Now we turn toward two SONC-based notions of convexity. In contrast to SOS-convexity, which admits several equivalent characterizations~\cite{ahmadi2013complete}, SONC-convexity does not enjoy such equivalences in general. This motivates the distinction between first- and second-order variants.

A related phenomenon has recently been observed in the setting of polyconvexity. Fantuzzi et al.~\cite{FantuzziEtAl2026Polyconvexity}
introduce SOS-based strengthenings of different characterizations of polyconvexity and show that these strengthenings need not remain
equivalent. Since polyconvexity concerns convexity after lifting to the matrix minors, rather than ordinary convexity in the original
variables, their framework is complementary to the SONC-based notions considered here.

\subsubsection{First-order SONC-convexity} \label{subsubsec: first-order SONC-convexity}
Recall that a polynomial $f$ is convex on $\R^n$ if and only if the following function \[
\struc{g_f(\xb,\yb)} \coloneqq f(\xb) - f(\yb) - \nabla f(\yb)^\top(\xb-\yb) \in \R[\xb; \yb]
\]
is globally nonnegative. This leads to the following definition.
\vspace{0.5em}
\begin{definition}\label{def:first_order_sonc}
    Let $f(\xb) \in\R[\xb]_{n,2d}$. We say that  $f$ is \struc{first-order SONC-convex} if
    \[
g_f(\xb,\yb) \in C_{2n,2d}.
\]
\end{definition}

The following are first-order SONC-convex polynomials:

\begin{example}\label{ex:1order_SONC-conv}
\leavevmode
    \begin{enumerate}
\item The quadratic form $f(\xb)=\frac{1}{2}\sum_{i=1}^n x_i^2$, since 
\[
g_f(\xb,\yb)=\frac{1}{2}\sum_{i=1}^n (x_i-y_i)^2
\]
is SONC.
\item More generally, separable quadratic polynomials 
\[
f(\xb)=\sum_{i=1}^n \big(a_i(x_i-\xi_i)^2 + b_i x_i\big) + c,\quad a_i\ge 0,
\]
satisfy $g_f(\xb,\yb)=\sum_{i=1}^n a_i(x_i-y_i)^2$, hence are SONC.

The special case $\xi_i = 0$ for every $i$, yields $f = \sum_{i=1}^n (a_i x_i^2 + b_ix_i) + c$, $a_i \geq 0$, which is a first-order SDSOS-convex polynomial, see \cite[Example 5.2~(i)]{chuong2019new}.
\item Let $d\geq 2$ and
\[
f(\xb)=\sum_{i=1}^n \big(a_i x_i^{2d}+b_i x_i^{2d-2}\big), \quad a_i,b_i\ge0.
\]
Then $f$ is first-order SONC-convex. 
Indeed,
\begin{align*}
g_f(\xb,\yb)
={}&
\sum_{i=1}^n
a_i\bigl(
x_i^{2d}
-2d\,x_i y_i^{2d-1}
+(2d-1)y_i^{2d}
\bigr)\\
&+
\sum_{i=1}^n
b_i\bigl(
x_i^{2d-2}
-(2d-2)x_i y_i^{2d-3}
+(2d-3)y_i^{2d-2}
\bigr).
\end{align*}
Each nonzero summand is a boundary nonnegative circuit polynomial, thus, $g_f \in C_{2n,2d}$.

In particular, the special case $f(\xb)=\sum_{i=1}^n a_i x_i^{2d}$ recovers a basic class of first-order SONC-convex polynomials, which is also discussed in context of first-order SDSOS-convexity in \cite[Example 5.2~(ii)]{chuong2019new}.
    \end{enumerate}
\end{example}

The examples in Example~\ref{ex:1order_SONC-conv} show that separable even-degree polynomials naturally admit SONC certificates via decompositions into univariate circuit components. 
This suggests that first-order SONC-convexity may be closely tied to separable structure. However, the next example shows that even a very simple nonseparable quadratic form is not captured by first-order SONC-convexity.

\begin{example}\label{ex:NOT1order_SONC-conv}
Consider
\[
f(x_1,x_2)=x_1^2 + x_2^2 + x_1 x_2.
\]

A direct computation gives
\begin{align*}
g_f(\xb,\yb) &= (x_1-y_1)^2 + (x_2-y_2)^2 + (x_1-y_1)(x_2-y_2)\\
&= x_1^2+x_2^2+x_1x_2+y_1^2+y_2^2+y_1y_2
-2x_1y_1-x_1y_2-x_2y_1-2x_2y_2.
\end{align*}
The monomial square coefficients sum to $4$, while the absolute values of the non-square coefficients sum to $8$. Hence the necessary condition of~\cite[Theorem~5]{dressler2025study} fails. Therefore, $f$ is not first-order SONC-convex.   
\end{example}

\bigskip
A particularly strong consequence of first-order SONC-convexity is that, within this class, nonnegativity and SONC membership coincide.

\vspace{0.5em}
\begin{theorem}
\label{Theorem: f being nonnegative iff f being SONC}
  Let $f\in \R[\xb]_{n,2d}$ be first-order SONC-convex. Then 
  \[
f \ge 0 \quad \text{ if and only if } \quad f \in C_{n,2d}.
\]
\end{theorem}  

\begin{proof}
The implication $f \in C_{n,2d} \Rightarrow f \ge 0$ is immediate.

Conversely, assume $f$ is nonnegative and first-order SONC-convex.
Then $f$ is convex and, being a convex polynomial of even degree that is bounded below, attains its global minimum at some $\ub$ with $\nabla f(\ub)=0$. By definition,
\[
g_f(\xb,\yb) \in C_{2n,2d}.
\]
By Lemma~\ref{lm:SONC_algebraic_closure} (iii), restricting $\yb=\ub$ yields
\[
g_f(\xb,\ub)=f(\xb)-f(\ub) \in C_{n,2d}.
\]
Since $f(\ub)\ge0$ and $C_{n,2d}$ is a cone, it follows that $f \in C_{n,2d}$.
\end{proof}

This notion will be instrumental for establishing another condition for first-level exactness of the bounded SOS+SONC hierarchy in Section~\ref{subsec: 1level_exactness_SONC-conv}.

\subsubsection{Second-order SONC-convexity}\label{Subsubsec: The second-order SONC-convexity}
We next introduce a Hessian-based notion of SONC-convexity.

\begin{definition}
\label{Definition: Second-order SONC-convexity}
    Let $f(\xb) \in \R[\xb]_{n,2d}$. Then $f$ is called \struc{second-order SONC-convex} if 
    \begin{align*}
        \struc{q_f(\xb,\yb)} \coloneqq \yb^\top H_f(\xb) \yb \in C_{2n,2d}.
    \end{align*}
\end{definition}

The following examples illustrate that second-order SONC-convexity is naturally satisfied by polynomials whose Hessian structure decomposes into circuit-compatible forms.
\begin{example}\label{ex:2order_SONC-conv}
\leavevmode
\begin{enumerate}
       \item Bivariate quadratic forms $f(x_1,x_2)=f(\xb)=\xb^\top Q \xb$ with $Q\succeq0$ are second-order SONC-convex, since 
    \[
q_f(\xb,\yb)=2\yb^\top Q \yb
\]
is a nonnegative quadratic in $\yb=(y_1,y_2)$, therefore $q_f \in \Sigma^h_{2,2}=C^h_{2,2}$ by~\cite[Theorem~3.1]{dressler2021real}.    
   \item We can generalize the bivariate case from (1) to the following result: \\  A quadratic form $f(\xb) \in \R[\xb]_{n,2}$ is second-order SONC-convex  if and only if it is  SONC. Indeed, if $f$ is SONC, then so is $q_f(\xb,\yb) = 2f(\yb)$. Conversely, if $q_f$ is SONC, then by restricting $\xb=0$ we obtain
$2f(\yb)\in C_{n,2}$, hence $f\in C_{n,2}$.
\item Let 
\[
f(x_1,x_2)=a(x_1^4+x_2^4)+b x_1^2 x_2^2,\quad a\ge0,\ 0\leq b\le6a.
\]
Then $f$ is second-order SONC-convex. Indeed, 
\[
q_f(\xb,\yb) = (12 ax_1^2y_1^2+12ax_2^2y_2^2+4bx_1x_2y_1y_2) + 2b (x_2y_1+x_1y_2)^2,
\]
where the first summand is a circuit polynomial with circuit number $24a$, so it is nonnegative when $4b\leq 24a$ and the second one is a binomial square.
Moreover, $b$ has to be nonnegative, since  otherwise $f$ is not convex because at $(x_1,x_2)=(0,1)$, the Hessian $H_f$ has diagonal entry $2b$.
\item For $f(\xb)=\sum_{i=1}^n a_i x_i^{2d}$ with $a_i\ge0$,
\[
q_f(\xb,\yb)=\sum_{i=1}^n a_i 2d(2d-1)x_i^{2d-2}y_i^2
\]
is a sum of monomial squares, hence SONC. In particular, this class is both first- and second-order SONC-convex, see Example~\ref{ex:1order_SONC-conv} (3).
    \end{enumerate}
\end{example}

These examples highlight that, in contrast to the first-order case, second-order SONC-convexity is governed by structural properties of the Hessian rather than direct decomposability of $f$ itself.
\medskip

For homogeneous polynomials, second-order SONC-convexity directly implies SONC.

\begin{theorem}\label{thm:second_order_form}
    For $2d\geq 2$, every second-order SONC-convex form is SONC.
\end{theorem}
\begin{proof}
Let $f$ be a form of degree $2d$. Euler's identity for homogeneous polynomials implies
  \begin{align*}
        \xb^\top \nabla f(\xb) = 2d\; f(\xb)
    \end{align*}
    and differentiating yields
\[
\xb^\top H_f(\xb)\xb = 2d(2d-1) f(\xb).
\]
Thus,
\[
f(\xb)=\frac{1}{2d(2d-1)} q_f(\xb,\xb).
\]
If $f$ is second-order SONC-convex, then $q_f(\xb,\yb)$ is SONC, and by closure under variable identification $\yb=\xb$ (Lemma~\ref{lm:SONC_algebraic_closure} (ii)), so is $q_f(\xb,\xb)$. The result follows.
\end{proof}

The above theorem is analogous to a classical result in SOS-convexity established in \cite{ahmadi2013complete}, which states that every SOS-convex form is SOS.

\bigskip
Second-order SONC-convexity is not stable under homogenization.

\begin{lemma}
\label{Lemma: 2nd SONC-convexity not closed under homogenization}
Second-order SONC-convexity is not preserved under homogenization.
\end{lemma}

\begin{proof}
Let $f$ be a second-order SONC-convex polynomial, and let $f^h$ denote the homogenization of $f$. By subtracting a sufficiently large positive constant, which does not affect second-order SONC-convexity, we can assume that $f$ is not globally nonnegative. Suppose, for contradiction, that $f^h$ is second-order SONC-convex. Then by Theorem~\ref{thm:second_order_form}, $f^h$ is SONC, and hence nonnegative on $\mathbb{R}^{n+1}$.
In particular, evaluating at $x_0=1$ yields
\[
f^h(1,\xb) = f(\xb) \ge 0,
\]
contradicting the existence of a point where $f(\xb) < 0$. Therefore second-order SONC-convexity is not preserved under homogenization.
\end{proof}

\medskip

We also record a useful sufficient condition extending Theorem~\ref{thm:second_order_form} to nonhomogeneous polynomials, under a mild normalization at the origin.

\begin{lemma}
    Let $f(\xb)$ be a nonnegative, second-order SONC-convex polynomial satisfying $\nabla f(0) = 0$. Then $f \in C_{n,2d}$.
\end{lemma}

\begin{proof}
Applying the second-order Taylor expansion of $f$ around the origin and invoking $\nabla f(0)=0$, leads to
\begin{align*}
    f(\xb) = f(0) + \int_0^1 (1-t)\xb^\top H_f(t\xb) \xb \, dt.
\end{align*}
Set 
\[
h_t(\xb)
\coloneqq
(1-t)\,\xb^\top H_f(t\xb)\xb .
\]
Since $f$ is second-order SONC-convex , we have $q_f(\xb,\yb)\in C_{2n,2d}$.
For every fixed $t\in[0,1]$, the polynomial
\[
h_t(\xb)=(1-t)\,q_f(t\xb,\xb)
\]
is obtained from $q_f$ by the nonnegative linear variable identification
$(\xb,\yb)\mapsto (t\xb,\xb)$ and multiplication by the nonnegative scalar
$1-t$. Hence, by Lemma~\ref{lm:SONC_algebraic_closure} (ii), we have $h_t\in C_{n,2d}$ for all $t\in[0,1]$. 
Moreover, since $H_f$ is a polynomial matrix, every coefficient of
$h_t(\xb)$ is a polynomial in $t$, and therefore Riemann integrable on
$[0,1]$. Thus Lemma~\ref{lm:SONC_integration} applies and gives
\[
\int_0^1 h_t(\xb)\,dt \in C_{n,2d}.
\]
Finally, invoking that $f$ is nonnegative, yields the claim.
\end{proof}

We briefly comment on the necessity of the normalization.
\begin{remark}\label{rem:2nd_order_SONC-conv_Nohomogeneous}
The normalization $\nabla f(0)=0$ in the previous lemma is essential for the above argument, and cannot be omitted in general. Indeed, without this assumption, the second-order integral representation yields
\[
f(\xb)=f(0)+\nabla f(0)^\top \xb + \int_0^1 (1-t)\,\xb^\top H_f(t\xb)\xb\,dt,
\]
where the integral term is SONC by second-order SONC-convexity and the closure properties established in Section~\ref{subsec:further study of the SONC cone}, and clearly the constant term $f(0)$ is SONC.
The remaining linear term, however, is not SONC unless it vanishes, since every SONC polynomial is globally nonnegative. Although such a linear term may sometimes be absorbed into a larger circuit polynomial, this is not possible in general.

For example, consider the quadratic polynomial
\[
        f(x_1,x_2)=(x_1+x_2-1)^2 .
\]
Then $f$ is nonnegative and second-order SONC-convex. Indeed, a direct computation yields
\[
        q_f(\xb,\yb)= 2(y_1+y_2)^2. 
\]
By~\cite[Theorem~3.1]{dressler2021real}, note that $q_f\in C^h_{2,2}$ since it is a homogeneous quadratic in $\yb$. See also Example~\ref{ex:2order_SONC-conv} (i).
However, $\nabla f(0,0)=(-2,-2)\neq 0$, and $f\notin C_{2,2}$. Thus the conclusion of the lemma fails without the normalization $\nabla f(0)=0$.
\end{remark}

\subsubsection{Relationship between the two notions}
We next compare the two notions of SONC-convexity. Both are motivated by classical characterizations of convexity, and both imply ordinary convexity.  However, they are not equivalent in general.

\smallskip

Using the second-order Taylor formula, we obtain
\[
g_f(\xb,\yb)=\int_0^1 (1-t)\, q_f\big(\yb+t(\xb-\yb),\,\xb-\yb\big)\,dt.
\]

This identity reveals the structural link between first- and second-order SONC-convexity: the first-order remainder $g_f$ can be expressed as an average of the quadratic forms defining second-order SONC-convexity.

However, this representation also explains the gap between the two notions. Even if $q_f(\xb,\yb)$ is SONC, the SONC cone is not stable under the substitution
\[
(\xb,\yb)\mapsto \big(\yb+t(\xb-\yb),\,\xb-\yb\big),
\]
and therefore SONC-ness of $q_f$ does not imply SONC-ness of $g_f$.

\medskip
In particular, second-order SONC-convexity does not imply first-order SONC-convexity, as the following example, adapted from \cite{chuong2019new},  demonstrates.

\begin{example}
\label{ex: Second-order but not first-order}

Consider again $f(x_1,x_2) = (x_1+x_2-1)^2$.   We have
    \begin{equation*}
        \begin{aligned}
        &g_f(\xb,\yb) = [(x_1+x_2)-(y_1+y_2)]^2, \quad
        &q_f(\xb,\yb) = 2(y_1+y_2)^2.
        \end{aligned}
    \end{equation*}
As shown in Remark~\ref{rem:2nd_order_SONC-conv_Nohomogeneous}, $f$ is second-order SONC-convex.  
On the other hand, $g_f$ has four square terms and six mixed terms, whose coefficients violate the necessary condition established in~\cite[Theorem~5]{dressler2025study}, for a polynomial to be SONC. Hence,
\[
g_f \notin C_{4,2}.
\]
Therefore, $f$ is not first-order SONC-convex.
\end{example}

Further examples of this phenomenon can be constructed, for instance
\[
f(x_1,x_2) = x_1^6 + x_2^6 + 2x_1^4x_2^2 + 2x_1^2x_2^4,
\]
which is second-order but not first-order SONC-convex. These examples confirm that, in general, second-order SONC-convexity does not imply first-order SONC-convexity.

\medskip
Next, we consider the converse direction. While the two notions are not equivalent in general, the reverse implication is only known in a low-dimensional special case.
\begin{lemma}
    When $(n,2d) = (1,4)$, every first-order SONC-convex polynomial is second-order SONC-convex.
\end{lemma}
\begin{proof}
Every first-order SONC-convex polynomial $f$ is convex and thus its Hessian is positive semidefinite. For $(n,2d) = (1,4)$ that means, that the second derivative of $f$ satisfies, $f''(x)\geq 0$ for all $x\in \R$. Since $P_{1,2}=C_{1,2}$ by~\cite[Theorem~3.1]{dressler2021real}, $f''$ is SONC. Multiplication by the monomial square $y^2$ shifts all exponents by an even vector and therefore preserves SONC membership. Then,  $q_f(x,y)=y^2f''(x) \in C_{2,4}$, and $f$ is second-order SONC-convex.
\end{proof}

Thus, the reverse implication holds in this special case, but whether it extends to the general setting remains open and would be an interesting direction for future work.

\medskip

From an optimization perspective, first-order SONC-convexity is the more directly useful notion, as it yields certificate-based characterizations of nonnegativity and underlies the additional one-step convergence results established hereafter in Section~\ref{subsec: 1level_exactness_SONC-conv}. 

From a computational perspective, however, second-order SONC-convexity can be easier to test, since it is formulated in terms of the Hessian and may lead to smaller REP certificates. Moreover, second-order SONC-convexity captures structural properties of the Hessian and provides complementary insight into the geometry of the SONC cone.

\subsection{First-level exactness via SONC-convexity}\label{subsec: 1level_exactness_SONC-conv}

We now return to the B-SOS+SONC hierarchy and show that first-order SONC-convexity yields a new condition for first-level exactness.

\begin{theorem}
\label{thm:SONC-convex-first-level}
Assume that Slater's condition holds. If $f$ and $-g_j$, $j=1,\dots,m$, are first-order SONC-convex polynomials of degree at most $2k$, then the B-SOS+SONC hierarchy is exact at the first level, i.e.,
    \[ p_1^k = f_K^{\star}. \]
\end{theorem}

\begin{proof}
Since first-order SONC-convexity implies ordinary convexity, $f$ is convex and each $g_j$ is concave. Assumption~A gives existence of an optimal solution, and Slater's condition yields corresponding KKT
multipliers.
Let $\xb^{\star}$ be a primal optimal solution and let $\lamb = (\lambda_1,\dots,\lambda_m) \in \mathbb{R}_{\geq 0}^m$ be corresponding KKT multipliers. Define the Lagrangian
\[
L(\xb) \coloneqq f(\xb) - f_K^{\star} - \sum_{j=1}^m \lambda_j g_j(\xb).
\]
Because the cone of first-order SONC-convex polynomials is closed under nonnegative linear combination, $L$ is first-order SONC-convex.
Hence, 
\begin{align*}
   g_L (\xb,\yb) = L(\xb) - L(\yb) - \nabla L(\yb)^\top (\xb - \yb) 
\end{align*}
is SONC.
Evaluating at $\yb = \xb^{\star}$ and invoking the KKT conditions $\nabla L(\xb^{\star})=0$ and $L(\xb^{\star})=0$, yields $L(\xb) = g_L(\xb,\xb^{\star}) \in C_{n,2k}$ using Lemma~\ref{lm:SONC_algebraic_closure} (iii).
This provides a valid certificate at level $d=1$ of the full-support B-SOS+SONC hierarchy, implying $p_1^k = f_K^{\star}$.
\end{proof}

This result provides a direct SONC-based mechanism for first-level exactness, distinct from SOS-convexity-based approaches. The certificate arises from the first-order structure of the SONC cone rather than semidefinite representability.

\smallskip
It would be interesting to identify more general convexity-type conditions tailored to the SOS+SONC framework that still guarantee first-level exactness, and to understand whether such conditions can be formulated in a unified way beyond the present first- and second-order notions.
As a first step in this direction note the following. 
\begin{remark}\label{rem:mixedConvexitycertificate}
The proof of Theorem~\ref{thm:SONC-convex-first-level} is not specific to a purely SONC certificate. Indeed, one may replace the condition
\[
g_f(\xb,\yb)\in C_{2n,2d}
\]
by the mixed condition
\[
g_f(\xb,\yb)\in(\Sigma+C)_{2n,2d}.
\]
Observe that Lemma~\ref{lm:SONC_algebraic_closure}(iii), together with the evident closure of the SOS cone under constant restriction,
implies that the SOS+SONC cone is preserved under constant restriction as well. Thus, the same KKT argument shows that, under Slater's condition and the same degree assumptions, the mixed conditions
\[
g_f,\;g_{-g_1},\ldots,g_{-g_m}
\in(\Sigma+C)_{2n,2k}
\]
also imply
\[
p_1^k=f_K^\star.
\]
A systematic study of this mixed convexity condition and its relationship with SOS-convexity and first-order SONC-convexity is left for future work.
\end{remark}

To conclude this section, note that the preceding first-level exactness results have been stated for the full-support B-SOS+SONC hierarchy. The same conclusions hold for a support-restricted relaxation under a natural support-compatibility condition.

\begin{remark}\label{rem:support-restricted-exactness}
More precisely, let $\cA\subseteq\N_{2k}^n$, and denote by $p_{1,\cA}^k$ the first-level relaxation obtained by replacing $(\Sigma+C)_{n,2k}$ with $(\Sigma+C)_{n,2k}(\cA)$.
Suppose that one of the settings of Corollary~\ref{cor:SOS-convex-first-level}, Theorem~\ref{thm:SONC-convex-first-level}, or
Remark~\ref{rem:mixedConvexitycertificate} holds. Let $(\xb^\star,\lamb^\star)$ be an optimal KKT pair, and define the corresponding KKT residual by
\[
h^\star(\xb)
\coloneqq
f(\xb)-f_K^\star
-\sum_{j=1}^m\lambda_j^\star g_j(\xb).
\] 
If $\supp(h^\star)\subseteq\cA$, then $p_{1,\cA}^k=f_K^\star$. 

Indeed, in the three respective settings, the corresponding proof gives
\[
h^\star\in\Sigma_{n,2k},\qquad
h^\star\in C_{n,2k},\qquad\text{or}\qquad
h^\star\in(\Sigma+C)_{n,2k}.
\]
The support therefore implies $h^\star\in(\Sigma+C)_{n,2k}(\cA)$.
Choosing the first-level Krivine-Stengle multipliers according to 
\[
\widehat{\lambda}_{\mathbf e_j,\mathbf 0}
=\lambda_j^\star,
\qquad j=1,\ldots,m,
\]
and setting all remaining multipliers equal to zero yields
\[
L_1(\xb,\widehat{\lamb})-f_K^\star
=h^\star(\xb).
\]
Thus, $f_K^\star$ is feasible for $p_{1,\cA}^k$, and the reverse inequality follows from the validity of the relaxation.

In particular, the readily verifiable condition
\[
\{\mathbf 0\}
\cup\supp(f)
\cup\bigcup_{j=1}^m\supp(g_j)
\subseteq\cA
\]
is sufficient for support-restricted first-level exactness. The support condition above may, however, be strictly weaker because cancellations can occur in the KKT residual.
\end{remark}

\section{Numerical Experiments}
\label{sec:numerics}

In this section, we illustrate the B-SOS+SONC relaxation and evaluate its computational performance. 
We first present simple examples to demonstrate the mechanism of the relaxation, then comment on its implementation, and finally report numerical results.

\subsection{Illustrative examples}
We begin with three small-scale examples highlighting the behavior of the relaxation.

\begin{example}
Consider the polynomial optimization problem
\begin{equation}
\label{eq:example1}
\left\{
\begin{array}{ll}
\min & f(x,y)=x^4y^2+x^2+y^4-3x^2y^2 \\[0.3em]
\mathrm{s.t.}
& g_1(x,y)=x\geq0,\\
& g_2(x,y)=y\geq0,\\
& g_3(x,y)=1-x^2-y^2\geq0.
\end{array}
\right.
\end{equation}
Thus, $K$ is the quarter unit disk. In particular,
$0\leq g_j\leq1$ on $K$ for $j=1,2,3$, and
Assumption~A is satisfied.

Fix $k=3$ and let $\cA=\{\alpb \in \N^2\quantify |\alpb| \leq 6\}$.
Then the B-SOS+SONC relaxation reads
\[
p_d^3
=
\sup\left\{
t\in\R \quantify
L_d(x,y,\lamb)-t
\in(\Sigma+C)_{2,6}(\cA),\
\lamb\geq0
\right\}.
\]

A direct circuit number computation shows that $f$ is a nonnegative circuit polynomial, i.e., $f\in C_{2,6}(\cA)$.
Since $f$ is globally nonnegative and $f(0,0)=0$, we obtain
\[
f_K^\star=0.
\]
At order $d=1$, choosing $\lamb=0$ and $t=0$ yields
\[
L_1(x,y,\lamb)-t=f(x,y)\in C_{2,6}(\cA).
\]
Thus, $p_1^3\geq0$. Since every feasible value of the relaxation is a
lower bound for $f_K^\star$, it follows that
\[
p_1^3=f_K^\star=0.
\]

We finally compare this example with the assumptions of
Theorem~\ref{thm:SONC-convex-first-level}. The constraint polynomials
do satisfy the required first-order SONC-convexity conditions. Indeed,
$-g_1$ and $-g_2$ are affine, while
\[
-g_3(x,y)=x^2+y^2-1
\]
has first-order remainder
\[
g_{-g_3}\bigl((x,y),(u,v)\bigr)
=
(x-u)^2+(y-v)^2\in C_{4,2}.
\]
However, the objective is not convex, since
\[
\nabla^2 f(0,1)
=
\begin{pmatrix}
-4&0\\
\phantom{-}0&12
\end{pmatrix}.
\]
Consequently, $f$ is not first-order SONC-convex, and
Theorem~\ref{thm:SONC-convex-first-level} does not apply. This example
therefore shows that first-level exactness may also occur beyond the
scope of that theorem.
\end{example}

The next example illustrates the first-level exactness mechanism of Theorem~\ref{thm:SONC-convex-first-level}.

\begin{example}\label{ex:first-level-exactness-SONC-convex}
Consider the polynomial optimization problem
\begin{equation}\label{eq:first-level-exactness-SONC-example}
\left\{
\begin{array}{ll}
\min & f(x_1,x_2)=x_1^4+x_2^4-5x_1+8 \\[0.3em]
\text{s.t.}
& g_1(\xb)=x_1\geq 0,\\
& g_2(\xb)=x_2\geq 0,\\
& g_3(\xb)=1-x_1-x_2\geq 0.
\end{array}
\right.
\end{equation}
The feasible set is the standard two-dimensional simplex. In particular, Assumption~A holds, and Slater's condition is satisfied, 
for instance at $\xb=\left(\frac{1}{3},\frac{1}{3}\right)$.

Since the constraint polynomials are affine, each polynomial $-g_j$ is first-order SONC-convex. Moreover, the affine terms in $f$ do not
contribute to its first-order remainder, and hence 
\begin{align*}
g_f(\xb,\yb)
=(x_1^4-4x_1y_1^3+3y_1^4) + (x_2^4-4x_2y_2^3+3y_2^4).
\end{align*}
Each summand is a nonnegative circuit polynomial whose circuit inequality holds with equality. Therefore,
$g_f(\xb,\yb)\in C_{4,4},$ and $f$ is first-order SONC-convex.

The global minimizer is $\xb^\star=(1,0)$, with
\[
f_K^\star=f(1,0)=4.
\]
Indeed, $4x_1^3-5<0$ for every $x_1\in[0,1]$, while $x_2^4\geq0$.
The corresponding KKT multipliers for $g_1,g_2,g_3$ are
\[
\lamb^\star=(0,1,1).
\]
Therefore,
\begin{align*}
f(\xb)-f_K^\star
-\sum_{j=1}^3\lambda_j^\star g_j(\xb)
=
f(\xb)-4-g_2(\xb)-g_3(\xb)
=
x_1^4-4x_1+3+x_2^4.
\end{align*}
The polynomial $x_1^4-4x_1+3$ is a boundary nonnegative circuit polynomial, and $x_2^4$ is a monomial square. Thus,
\[
f(\xb)-f_K^\star
-\sum_{j=1}^3\lambda_j^\star g_j(\xb)
\in C_{2,4}.
\]
Hence, for $k=2$, the B-SOS+SONC hierarchy is exact at the first level:
\[
p_1^2=f_K^\star=4.
\]
\end{example}
Example~\ref{ex:first-level-exactness-SONC-convex} illustrates how the KKT multipliers combine the objective and the constraint polynomials
to produce an explicit SONC certificate at the first level.
\smallskip

The final example considers an objective polynomial that genuinely requires both components of the SOS+SONC cone.

\begin{example}\label{ex:mixed-SOS-SONC}
Consider the polynomial optimization problem
\begin{equation}\label{eq:mixed-SOS-SONC-example}
\left\{
\begin{array}{ll}
\min
& f(x,y,z)=(xy+xz+yz)^2+Q(x,y,z)+1 \\[0.3em]
\mathrm{s.t.}
& g_1(x,y,z)=x\geq0,\\
& g_2(x,y,z)=y\geq0,\\
& g_3(x,y,z)=z\geq0,\\
& g_4(x,y,z)=1-x^2-y^2-z^2\geq0,
\end{array}
\right.
\end{equation}
where
\[
\struc{Q(x,y,z)}
\coloneqq
x^2y^2+x^2z^2+y^2z^2+1-4xyz
\]
is the Choi-Lam polynomial. Thus, $K$ is the intersection of the nonnegative orthant with the unit ball. In particular,
$0\leq g_j\leq1$ on $K$ for $j=1,\ldots,4$, and Assumption~A is satisfied.

The polynomial $Q$ is a nonnegative circuit polynomial on the boundary of $C_{3,4}$.
Since $(xy+xz+yz)^2$ is a square and the constant $1$ is a monomial square, we obtain the explicit decomposition
\[
f
=
\underbrace{(xy+xz+yz)^2}_{\in\Sigma_{3,4}}
+
\underbrace{Q}_{\in C_{3,4}}
+
1
\in(\Sigma+C)_{3,4}.
\]
In fact,
\[
f\in
(\Sigma+C)_{3,4}
\setminus
\bigl(\Sigma_{3,4}\cup C_{3,4}\bigr),
\]
see~\cite[Lemma~4.2.8]{Schick2025-01-31squar-72682}.

Fix $k=2$ and let $\cA=\N_4^3$. The corresponding B-SOS+SONC relaxation is
\[
p_d^2
=
\sup\left\{
t\in\R\quantify
L_d(x,y,z,\lamb)-t
\in(\Sigma+C)_{3,4}(\cA),\
\lamb\geq0
\right\}.
\]
At level $d=1$, choosing $\lamb=0$ and $t=1$ yields
\[
L_1(x,y,z,\lamb)-1
=
f(x,y,z)-1
=
(xy+xz+yz)^2+Q(x,y,z)
\in(\Sigma+C)_{3,4}(\cA).
\]
As a consequence,
\[
p_1^2\geq1.
\]

The global optimal value can be determined exactly. Set
\[
\sigma_1=x+y+z,\qquad
\sigma_2=xy+xz+yz,\qquad
\sigma_3=xyz.
\]
Since
\[
x^2y^2+x^2z^2+y^2z^2
=
\sigma_2^2-2\sigma_1\sigma_3,
\]
we have
\[
f
=
2\sigma_2^2
-2\sigma_1\sigma_3
-4\sigma_3
+2.
\]
For $x,y,z\geq0$,
\[
\sigma_2^2
=
(xy+xz+yz)^2
\geq
3xyz(x+y+z)
=
3\sigma_1\sigma_3.
\]
It follows that
\[
f\geq2+4\sigma_3(\sigma_1-1).
\]
If $\sigma_1\geq1$, this gives $f\geq2$. If
$0\leq\sigma_1\leq1$, the arithmetic-geometric mean inequality gives
\[
\sigma_3\leq\left(\frac{\sigma_1}{3}\right)^3,
\]
and hence
\begin{align*}
f
\geq
2-\frac{4}{27}\sigma_1^3(1-\sigma_1)\geq
2-\frac{4}{27}\cdot\frac{27}{256}
=
\frac{127}{64},
\end{align*}
where we used
\[
\max_{0\leq s\leq1}s^3(1-s)=\frac{27}{256}.
\]
Equality is attained at
\[
(x^\star,y^\star,z^\star)
=
\left(\frac14,\frac14,\frac14\right),
\]
which belongs to $K$. Therefore,
\[
f_K^\star
=
\frac{127}{64}
=
1.984375.
\]
This exact value agrees with the numerical value $1.9844$ certified by GloptiPoly~3.

Numerically (see Table~\ref{tab:qdk-pdk-part2}, function $f_8$), the B-SOS+SONC implementation returns
\[
p_1^2\approx1.6574,
\]
whereas the corresponding first-level B-SOS bound is
\[
q_1^2\approx-24.6937.
\]
Thus, although the B-SOS+SONC relaxation is not exact at the first level for this problem, the additional SONC component produces a
substantially stronger lower bound.
\end{example}

The three examples illustrate complementary aspects of the B-SOS+SONC relaxation. The first example exhibits first-level exactness through an explicit SONC certificate, even though the objective does not satisfy the convexity assumptions of Theorem~\ref{thm:SONC-convex-first-level}. The second example showcases the first-level exactness mechanism guaranteed by that theorem. Finally, Example~\ref{ex:mixed-SOS-SONC} considers an
objective that belongs to the SOS+SONC cone but to neither constituent cone individually, and shows that the combined relaxation may
substantially improve the corresponding B-SOS bound already at low relaxation order.

\subsection{Implementation}
We implement the B-SOS+SONC relaxation in MATLAB, based on the modeling language YALMIP~\cite{yalmip}. Our solver of choice is MOSEK~\cite{mosek} as it allows to optimize over both the SDP and the REP cone. The code extends the SOS+SONC toolbox of Schick~\cite{sosplussonctoolbox}, originally designed for unconstrained polynomial optimization, to the constrained setting considered here. The implementation and all test instances are publicly available at 
\medskip

\centerline{\url{https://github.com/wenky2002/B_SOS_plus_SONC}}

\subsection{Computational experiments}
\label{subsec: Computational experiments}
We compare the (sparse) B-SOS, the B-SOS+SONC, and  Lasserre's hierarchies on a collection of polynomial optimization test problems, and solve them using the code for the sparse B-SOS~\cite{weisser2018sparse} relaxation, our toolbox, and \linebreak\texttt{gloptipoly3}~\cite{gloptipoly3}, respectively. 
For the B-SOS+SONC bounds, we use the full ambient support $\cA=\N^n_{2k}$ at certificate degree $k$.
The resulting lower bounds at the indicated relaxation orders are collected in Tables~\ref{tab:qdk-pdk} and~\ref{tab:qdk-pdk-part2}. As indicated, displayed generators are used with the two-sided normalization $0\leq g_j\leq1$; equivalently, the defining inequalities are $g_j\geq0$ and $1-g_j\geq0$. For Lasserre's hierarchy, we additionally report the first relaxation order at which the optimal value was obtained and successfully certified. 
All computations were executed in MATLAB R2022b on a laptop with a 2.5GHz Intel Core i5-1235U and 16GB of RAM.

\medskip

These numerical results illustrate the complementary roles of the SOS and SONC components. For the examples whose objectives are SONC but not SOS, the B-SOS+SONC relaxation can substantially improve upon B-SOS already at low relaxation orders. This is particularly apparent for $f_1$, $f_2$, $f_6$, and $f_{12}=m$, where the B-SOS bounds remain relatively far from the optimal value while the B-SOS+SONC bounds are essentially optimal. For $f_{12}$, for instance, the B-SOS bound at order $3$ 
is approximately $0.0005$, whereas the B-SOS+SONC already reaches the optimal value $0.5$ at order $2$. 
Similarly, for the examples combining SOS and SONC structure, such as $f_8$ and $f_9$, the hybrid relaxation improves the B-SOS bounds progressively and reaches values close to the corresponding Lasserre bounds at relatively low orders. Note that we did not compute higher order B-SOS+SONC relaxations for $f_8$ than $6$, since the computation of $p_6^2$ took $103.32$ hours. 
On the other hand, for instances with predominantly SOS structure, the two relaxations can behave similarly, as illustrated by $f_5$ and $f_{11}$.

The experiments also indicate some numerical limitations of the bounded-degree relaxations. In several instances, the higher-order B-SOS relaxations become numerically unstable: for $f_2$, $f_4$, $f_8$, $f_{10}$, and $f_{12}$, the final reported B-SOS relaxation exhibits coefficients or objective values of very large magnitude (order $O(10^{\ell})$ with $\ell\gg 1$), indicating a breakdown of numerical stability rather than meaningful growth of the relaxation bounds. In the case of $f_{10}$, the computed B-SOS sequence is moreover not monotone, despite the theoretical monotonicity of the hierarchy. This is another indication that the apparent behavior at these higher orders is caused by numerical inaccuracies. No analogous instability was observed within the B-SOS+SONC orders reported in the tables. 

For comparison, GloptiPoly~3 successfully finds and verifies the optimal value for all test instances except $f_9$. In particular, the reported Lasserre bounds attain the optimal value at order one for most examples, while higher orders are required for $f_2$, $f_8$, and $f_9$. For $f_9$, GloptiPoly~3 does not certify optimality within the orders considered, although the computed lower bound is already very close to the optimal value.

\medskip

Overall, the numerical results show that incorporating SONC structure into bounded SOS hierarchies can yield systematically tighter relaxations while preserving the tractability of SDP-based methods, making the proposed approach a promising alternative to purely SOS-based methods.

\clearpage
\begin{landscape}
\begin{table}[ht]
  \centering
  \captionsetup{justification=centering, margin=1cm}
   \caption{Comparison of sparse B-SOS, B-SOS+SONC, and Lasserre's hierarchy}
  \label{tab:qdk-pdk}
  \scriptsize
  \renewcommand{\arraystretch}{1.3}
  \setlength{\tabcolsep}{4pt}
  \setlength{\aboverulesep}{-0.3pt}
  \setlength{\belowrulesep}{-0.3pt}
  \setlength{\heavyrulewidth}{1.1pt}
  \setlength{\lightrulewidth}{1pt}
  \begin{tabularx}{\linewidth}{@{}
    >{\RaggedRight\arraybackslash}p{0.16\linewidth}
    >{\centering\arraybackslash}p{0.12\linewidth}
    >{\RaggedRight\arraybackslash}p{0.11\linewidth}
    @{\hspace{6pt}}
    !{\color{gray!60}\vrule width 0.8pt}
    @{\hspace{6pt}}
    >{\RaggedRight\arraybackslash}p{0.17\linewidth}
    @{\hspace{1pt}}
    >{\RaggedRight\arraybackslash}p{0.16\linewidth}
    @{\hspace{6pt}}
    >{\centering\arraybackslash}p{0.11\linewidth}
    >{\centering\arraybackslash}X
    @{}}
  \toprule
  \multicolumn{1}{l}{\multirow{2}{*}{\textbf{Objective function}}}
  & \multirow{2}{*}{\makecell[c] {\textbf{Cone}\\ \textbf{membership}}}
  & \multirow{2}{*}{\makecell[c]{\textbf{Constraint} \\[2pt] $0 \leq g_j \leq 1$}}
  & \multicolumn{1}{l}{\multirow{2}{*}{\textbf{B-SOS}}}
  & \multicolumn{1}{l}{\multirow{2}{*}{\textbf{B-SOS+SONC}}}
  & \multicolumn{2}{c}{\textbf{Lasserre's hierarchy}} \\
  \cmidrule(l){6-7}
  & & & & & \textbf{Value} & \makecell[c]{\textbf{Order,}\\\textbf{Optimal}} \\
  \midrule
  $f_1=x^4y^2 \!+\! x^2 \!+\! y^4 \!-\! 3x^2y^2$
  & $C_{2,6} \setminus \Sigma_{2,6}$
  & $\begin{aligned}[t]
      g_1 &= x \\
      g_2 &= y \\
      g_3 &= x^2\!+\!y^2
    \end{aligned}$
  & $\begin{array}[t]{r@{{}={}}l}
      q_1^3 & -0.4034 \\
      q_2^3 & -0.2633 \\
      q_3^3 & -0.0258 \\
      q_4^3 & -0.0105 \\
      q_5^3 & -3.2369\!\times\!10^{-4}
    \end{array}$
  & $\begin{array}[t]{r@{{}={}}l}
      p_1^3 & -2.4566\!\times\! 10^{-9}
   \end{array}$
  & $1.6157 \!\times\! 10^{-7}$ 
  & \text{$1$, yes}
  \\[4pt]
\arrayrulecolor{gray!35}\specialrule{0.4pt}{0pt}{0pt}
\arrayrulecolor{black}

  $f_2=x^4y^2\!+\!x^2y^4\!+\!z^6\!-\!3x^2y^2z^2$
  & $C_{3,6}\setminus \Sigma_{3,6}$
  & $\begin{aligned}[t]
      g_1 &= x \\
      g_2 &= y \\
      g_3 &= z \\
      g_4 &= x^2\!+\!y^2\!+\!z^2
    \end{aligned}$
  & $\begin{array}[t]{r@{{}={}}l}
      q_1^3 & -1.3371 \\
      q_2^3 & -1.2950 \\
      \multicolumn{2}{c}{\vdots}\\
      q_5^3 & -1.1676 \\
      q_6^3 & \text{unreliable}^{\dagger}
    \end{array}$
    & $\begin{array}[t]{r@{{}={}}l}
      p_1^3 & -4.8238\!\times\! 10^{-9}
   \end{array}$
     & $-1.8905 \!\times\! 10^{-8}$
     & \text{$4$, yes}
    \\[4pt]
\arrayrulecolor{gray!35}\specialrule{0.4pt}{0pt}{0pt}
\arrayrulecolor{black}

  $f_3= 1 \!-\! x^2y^2 \!+\! x^2y^4 \!+\! x^4y^2$
  & $C_{2,6}\setminus\Sigma_{2,6}$
  & $\begin{aligned}[t]
      g_1 &= x \\
      g_2 &= y
    \end{aligned}$
  & $\begin{array}[t]{r@{{}={}}l}
      q_1^3 & 0.9575 \\
      q_2^3 & 0.9610 \\
      \multicolumn{2}{c}{\vdots}\\
      q_7^3 & 0.9629 \\
      q_8^3 & 0.9630
    \end{array}$
  & $\begin{array}[t]{r@{{}={}}l}
      p_1^3 & 0.9630
   \end{array}$
  & $0.9630$
  & \text{$1$, yes}
    \\[4pt]
\arrayrulecolor{gray!35}\specialrule{0.4pt}{0pt}{0pt}
\arrayrulecolor{black}

  $f_4=x^{6}\!+\!y^{6}\!+\!1 \!+\!3x^{2}y^{2} \!-\!(x^{4}y^{2}\!+\!x^{4}\!+\!x^{2}y^{4}\!+\!x^{2}\!+\!y^{4}\!+\!y^{2})$
  & $ P_{2,6} \setminus (\Sigma\!+\!C)_{2,6}$
  & $\begin{aligned}[t]
      g_1 &= x \\
      g_2 &= y
    \end{aligned}$
  & $\begin{array}[t]{r@{{}={}}l}
      q_1^3 & -1.8619 \\
      q_2^3 & -1.1450 \\
      q_3^3 & -1.0312 \\
      q_4^3 & -1.0030 \\
      q_i^3 & -1.0000,\; i\in[5,9]\\
      q_{10}^3 & -0.9946\\
      \multicolumn{2}{c}{\vdots}\\
      q_{15}^3 & -0.9941\\
      q_{16}^3 & \text{unreliable}
    \end{array}$
  & $\begin{array}[t]{r@{{}={}}l}
      p_1^3 & -0.4433 \\
      p_2^3 & -0.0535 \\
      p_3^3 & -0.0111 \\
      p_4^3 & -4.8137 \!\times\! 10^{-4} \\
    \end{array}$
  & $-1.7985 \!\times\! 10^{-7}$
  & \text{$1$, yes}
    \\[4pt]
\arrayrulecolor{gray!35}\specialrule{0.4pt}{0pt}{0pt}
\arrayrulecolor{black}

  $f_5=(x\!-\!1)^2(x\!-\!2)^2$
  & $ \Sigma_{1,4} \setminus C_{1,4}$
  & $g_1=x$
  & $\begin{array}[t]{r@{{}={}}l}
  q_1^2 & -7.8595\!\times\! 10^{-9}\end{array}$
  & $\begin{array}[t]{r@{{}={}}l}
  p_1^2 & -3.2012\!\times\! 10^{-9}
  \end{array}$
  & $7.6576 \!\times\! 10^{-8}$
  & \text{$1$, yes}
    \\[4pt]
\arrayrulecolor{gray!35}\specialrule{0.4pt}{0pt}{0pt}
\arrayrulecolor{black}

  $Q=x^2y^2\!+\!x^2z^2\!+\!y^2z^2\!+\!1\!-\!4xyz$
  & $ C_{3,4} \setminus \Sigma_{3,4}$
  & $\begin{aligned}[t]
      g_1 &= x \\
      g_2 &= y \\
      g_3 &= z
    \end{aligned}$
  & $\begin{array}[t]{r@{{}={}}l}
      q_1^2 & -42.1307 \\
      q_2^2 & -0.0833 \\
      q_3^2 & 3.0982\!\times\! 10^{-7}
    \end{array}$
  & $\begin{array}[t]{r@{{}={}}l}
  p_1^2 & -1.6300\!\times\! 10^{-8}
  \end{array}$
  & $2.2262 \!\times\! 10^{-8}$
  & \text{$1$, yes}
    \\[4pt]
\arrayrulecolor{gray!35}\specialrule{0.4pt}{0pt}{0pt}
\arrayrulecolor{black}

  $m = x^4y^2 \!+\! x^2y^4\! +\! 1 \!- \!3x^2y^2$
  & $C_{2,6} \setminus \Sigma_{2,6}$
  & $\begin{aligned}[t]
      g_1 &= x \\
      g_2 &= y
    \end{aligned}$
  & $\begin{array}[t]{r@{{}={}}l}
      q_1^3 & -0.0713 \\
      q_2^3 & -0.0059 \\
      q_3^3 & 1.5019\!\times\! 10^{-7}
    \end{array}$
  & $\begin{array}[t]{r@{{}={}}l}
  p_1^3 & -2.0588\!\times\! 10^{-8}\end{array}$
  & $4.6032 \!\times\! 10^{-7}$
  & \text{$1$, yes}
  \\
  \bottomrule 
  \end{tabularx}
  
  \par\smallskip
  \begin{minipage}{\linewidth}
  	\footnotesize
  	$^\dagger$The solver returned a value of very large magnitude;
  	we regard this as an unreliable numerical output rather than a
  	valid lower bound.
  \end{minipage}
\end{table}
\end{landscape}

\clearpage
\begin{landscape}
\begin{table}[ht]
  \centering
  \captionsetup{justification=centering, margin=1cm}
  \caption{Comparison of sparse B-SOS, B-SOS+SONC, and Lasserre's hierarchy}
  \label{tab:qdk-pdk-part2}
  \scriptsize
  \renewcommand{\arraystretch}{1.3}
  \setlength{\tabcolsep}{4pt}
  \setlength{\aboverulesep}{-0.3pt}
  \setlength{\belowrulesep}{-0.3pt}
  \setlength{\heavyrulewidth}{1.1pt}
  \setlength{\lightrulewidth}{1pt}
  \begin{tabularx}{\linewidth}{@{}
    >{\RaggedRight\arraybackslash}p{0.16\linewidth}
    >{\centering\arraybackslash}p{0.16\linewidth}
    >{\RaggedRight\arraybackslash}p{0.09\linewidth}
    @{\hspace{6pt}}
    !{\color{gray!60}\vrule width 0.8pt}
    @{\hspace{6pt}}
    >{\RaggedRight\arraybackslash}p{0.18\linewidth}
    @{\hspace{1pt}}
    >{\RaggedRight\arraybackslash}p{0.13\linewidth}
    @{\hspace{6pt}}
    >{\centering\arraybackslash}p{0.11\linewidth}
    >{\centering\arraybackslash}X
    @{}}
\toprule
\multicolumn{1}{l}{\multirow{2}{*}{\textbf{Objective function}}}
& \multirow{2}{*}{\makecell[c] {\textbf{Cone}\\ \textbf{membership}}}
& \multirow{2}{*}{\makecell[c]{\textbf{Constraint} \\[2pt] $0 \leq g_j \leq 1$}}
& \multicolumn{1}{l}{\multirow{2}{*}{\textbf{B-SOS}}}
& \multicolumn{1}{l}{\multirow{2}{*}{\textbf{B-SOS+SONC}}}
& \multicolumn{2}{c}{\textbf{Lasserre's hierarchy}} \\
\cmidrule(l){6-7}
  & & & & & \textbf{Value} & \makecell[c]{\textbf{Order,}\\\textbf{Optimal}} \\
 \midrule

  $f_8=(xy \!+\! xz \!+\! yz)^2\! + \!1 \!+\! Q$
  & $(\Sigma\!+\!C)_{3,4} \setminus  (\Sigma_{3,4}\!\cup\!C_{3,4})$
  & $\begin{aligned}[t]
      g_1 &= x \\
      g_2 &= y \\
      g_3 &= z \\
      g_4 &= x^2\!+\!y^2\!+\!z^2
    \end{aligned}$
  & $\begin{array}[t]{r@{{}={}}l}
      q_1^2 & -24.6937 \\
      q_2^2 & 1.6063 \\
      q_3^2 & 1.7054 \\
      q_4^2 & 1.9595 \\
      q_5^2 & 1.9694 \\
      q_6^2 & \text{unreliable}
    \end{array}$
  & $\begin{array}[t]{r@{{}={}}l}
      p_1^2 & 1.6574 \\
      p_2^2 & 1.8499 \\
      p_3^2 & 1.9499 \\
      p_4^2 & 1.9677 \\
      p_5^2 & 1.9708\\
      p_6^2 & 1.9760
    \end{array}$
  & $1.9844$
  & \text{$3$, yes}
  \\[4pt]
  \arrayrulecolor{gray!35}\specialrule{0.4pt}{0pt}{0pt}
\arrayrulecolor{black}

  $f_9=\frac{1}{2}(1 \!+\! 2xy \!+\! x^2y)^2 \!+\! 2m$
  & $(\Sigma\!+\!C)_{2,6} \setminus (\Sigma_{2,6}\!\cup\!C_{2,6})$
  & $\begin{aligned}[t]
      g_1 &= x \\
      g_2 &= y \\
      g_3 &= x^2\!+\!y^2
    \end{aligned}$
  & $\begin{array}[t]{r@{{}={}}l}
      q_1^3 & 0.4805 \\
      q_2^3 & 2.2750 \\
      q_3^3 & 2.4061 \\
      q_4^3 & 2.4992\\
      q_5^3 & 2.5000
    \end{array}$
  & $\begin{array}[t]{r@{{}={}}l}
      p_1^3 & 1.0129 \\
      p_2^3 & 2.2800 \\
      p_3^3 & 2.4473\\
      p_4^3 & 2.4998\\
      p_5^3 & 2.5000
    \end{array}$
  & $2.4999$
  & \begin{array}[t]{r@{{},{}}l}
      \leq 15 & \text{unknown} \\
       16 & \text{no}
  \end{array}
  \\[4pt]
  \arrayrulecolor{gray!35}\specialrule{0.4pt}{0pt}{0pt}
\arrayrulecolor{black}

  $f_{10}= m \!+\! x^2$
  & $ C_{2,6} \setminus \Sigma_{2,6}$
  & $\begin{aligned}[t]
      g_1 &= x \\
      g_2 &= y \\
      g_3 &= x^2\!+\!y^2
    \end{aligned}$
  & $\begin{array}[t]{r@{{}={}}l}
      q_1^3 & 0.7058 \\
      q_i^3 & 0.7500,\; i\in[2,6]\\
      q_7^3 & 0.7382\\
      q_8^3 & 0.7301\\
      q_9^3 & \text{unreliable}
    \end{array}$
  & $\begin{array}[t]{r@{{}={}}l}
      p_1^3 & 0.8624 \\
      p_2^3 & 0.8750
    \end{array}$
  & $0.8750$
  & \text{$1$, yes}
  \\[4pt]
  \arrayrulecolor{gray!35}\specialrule{0.4pt}{0pt}{0pt}
\arrayrulecolor{black}

  $f_{11}= (1\!-\!x^2)^2 \!+\! (1\!-\!y^2)^2 \!+\! (1\!-\!z^2)^2 \!+\! Q$
  & $\Sigma_{3,4} \cap C_{3,4}$
  & $\begin{aligned}[t]
      g_1 &= x \\
      g_2 &= y \\
      g_3 &= z
    \end{aligned}$
  & $\begin{array}[t]{r@{{}={}}l}
  q_1^3 & -2.0552\!\times\! 10^{-9}
  \end{array}$
  & $\begin{array}[t]{r@{{}={}}l}
  p_1^3 & -1.6862\!\times\! 10^{-8}
  \end{array}$
  & $1.9448\! \times\! 10^{-8}$
  & \text{$1$, yes}
  \\[4pt]
  \arrayrulecolor{gray!35}\specialrule{0.4pt}{0pt}{0pt}
\arrayrulecolor{black}

  $ m = x^4y^2 \!+\! x^2y^4 \!+\! 1 \!-\! 3x^2y^2$
  & $ C_{2,6} \setminus \Sigma_{2,6}$
  & $\begin{aligned}[t]
      g_1 &= x \\
      g_2 &= y \\
      g_3 &= x^2\!+\!y^2
    \end{aligned}$
  & $\begin{array}[t]{r@{{}={}}l}
      q_1^3 & -3.3154 \\
      q_2^3 & -0.0103 \\
      q_3^3 & 4.8756\!\times\! 10^{-4} \\
      q_4^3 & 0.0102 \\
      \multicolumn{2}{c}{\vdots} \\
      q_8^3 & 0.0498 \\
      q_9^3 & \text{unreliable}
    \end{array}$
  & $\begin{array}[t]{r@{{}={}}l}
      p_1^3 & 0.4375 \\
      p_2^3 & 0.5000
    \end{array}$
  & $0.5000$
  & \text{$1$, yes}
  \\
  \bottomrule
  \end{tabularx}
\end{table}
\end{landscape}

\section{Conclusion}\label{sec:conclusion}

In this paper, we introduced the bounded-degree SOS+SONC hierarchy for constrained polynomial optimization, thereby combining the bounded-degree philosophy of Lasserre’s B-SOS relaxation with the additional expressive power of SONC certificates. The resulting framework preserves tractability through an SDP-REP reformulation, while enlarging the certificate space beyond SOS alone. We proved that the hierarchy is complete for every fixed certificate degree and that it leads to bounds that are at least as strong as those of the corresponding B-SOS relaxation. In addition, we initiated the study of the theory of SONC-convexity and showed that it gives rise to a new SONC-based first-level exactness mechanism for the B-SOS+SONC hierarchy, distinct from the classical SOS-convexity argument. 

Our numerical experiments confirm that the proposed hierarchy is effective in practice: it often improves upon B-SOS bounds already at low relaxation orders, while remaining computationally manageable. These results suggest that the combined SOS+SONC perspective is a promising direction for polynomial optimization, especially in situations where algebraic and sparse structures interact.

\smallskip
Several open questions remain.
First, a stopping criterion for B-SOS+SONC would be highly desirable, both from a theoretical and a computational point of view. In the SOS setting, such criteria are typically tied to dual moment formulations and flat extension phenomena; developing an analogous sound dual moment theory for SONC and SOS+SONC therefore appears to be a key step toward understanding both stopping criteria and solution recovery. For signomial optimization, a dual perspective for the related SAGE cone was already initiated by the first author in~\cite{Dressler:Murray}, where basic facts on the existence and uniqueness of solutions to signomial moment problems are established. Moreover, SAGE-based REP relaxations can also be used for solution recovery in signomial optimization; see~\cite[Section~3.2]{MCW2019}. However, a more streamlined SOS-like dual theory remains an important open direction.

Second, the role of sparsity in the combined SOS+SONC setting deserves further investigation. While our current support choice is convenient, it may be larger than necessary, especially for the SONC part, and does not fully exploit the sparsity-preserving property of SONC decompositions. 
In particular, our formulation~\eqref{eq:SOS+SONC_A-restricted_decomp} imposes a common support structure on the SOS and SONC parts, whereas there is no reason to expect that an optimal decomposition should satisfy such a restriction. Allowing the SOS and SONC summands to have distinct, problem-adapted supports could potentially lead to substantially smaller SDP-REP formulations.
A sharper understanding of how the supports of the SOS and SONC summands can be controlled in terms of the original polynomial could lead to a genuinely sparse version of the hierarchy and, in turn, to substantially larger-scale methods.

Third, it would be interesting to study whether the SONC cone remains closed under further natural restrictions, beyond the variable identifications and constant restrictions considered here, as this would shed more light on the structural robustness of the cone. 

Finally, a systematic study of convexity-type and other structural conditions adapted to the SOS+SONC cone that guarantee first-level
exactness remains an interesting direction for future work.
This includes clarifying the relationship between first- and second-order SONC-convexity. While we showed that second-order SONC-convexity does not imply first-order SONC-convexity, it is currently unknown whether the converse implication holds in general.
Another natural question is whether the mixed SOS+SONC condition discussed in Remark~\ref{rem:mixedConvexitycertificate} defines a genuinely larger class than the corresponding SOS- and SONC-based conditions, and whether it admits computationally effective tests.
More broadly, it would be interesting to determine whether these notions, together with further structural conditions for first-level exactness, can be embedded into a unified framework for certificate-based convexity.

We believe that these questions will be central for turning the present framework into a broader and more practical optimization methodology.

\medskip

\section*{Acknowledgments}
Mareike Dressler is supported by the Australian Research Council Discovery Early Career Award DE240100674. 
We thank Guoyin Li for valuable input and helpful comments during the early stages of this project. 
\medskip


\bibliographystyle{alpha}
\bibliography{references}

\end{document}